\documentclass[a4paper,twoside,10pt]{article}

\usepackage[T1]{fontenc}
\usepackage[latin1]{inputenc}

\usepackage{amsmath}
\usepackage{amsthm}
\usepackage{amsfonts}
\usepackage{enumerate}
\usepackage{a4wide}
\usepackage{amsfonts}
\usepackage{amssymb}
\usepackage{amsthm}
\usepackage{latexsym}

\usepackage[active]{srcltx}

\usepackage[latin1]{inputenc}

\usepackage[active]{srcltx}
\newtheorem{theorem}{Theorem}[section]

\newtheorem{defi}[theorem]{Definition}

\newtheorem{lemma}[theorem]{Lemma}
\newtheorem{remark}[theorem]{Remark}
\newtheorem{proposition}[theorem]{Proposition}

\newtheorem{corollary}[theorem]{Corollary}

\usepackage{setspace}

\def\N{{\mathbb N}}

\def\Z{{\mathbb Z}}

\def\N{\mathbb N}

\def\la{\lambda}

\def\la{\lambda}

\def\ps{\psi}
\def\N{\mathbb{N}}

\def\nn{\nonumber}

\def\N{{\mathbb N}}

\def\K{{\mathbb K}}

\date{}
\begin{document}

\title{Associative superalgebras with homogeneous symmetric structures}
\maketitle
\begin{center}
\author{Imen Ayadi and Sa{\"i}d Benayadi}
\end{center}
\begin{center}
Laboratoire de Math\'ematiques et Applications de Metz, CNRS-UMR 7122, Universit\'e Paul Verlaine-Metz, Ile du Saulcy, F-57045 Metz Cedex 1, FRANCE.\\
\end{center}
\begin{center} 
imen.ayadi@univ-metz.fr; benayadi@univ-metz.fr;  
\end{center}

\begin{abstract}
A homogeneous symmetric structure on an associative superalgebra  $A$ is a non-degenerate, supersymmetric, homogeneous (i.e. even or odd) and associative bilinear form on $A$. In this paper, we show that any associative superalgebra with non null product can not admit simultaneously even-symmetric and odd-symmetric structure. We prove that all simple associative superalgebras admit either even-symmetric or odd-symmetric structure and we  give explicitly, in every case, the homogeneous symmetric structures. We introduce some notions of generalized double extensions in order to give inductive descriptions of even-symmetric associative superalgebras and odd-symmetric associative superalgebras. We obtain also an other interesting description of odd-symmetric associative superalgebras whose even parts are semi-simple bimodules without using the notions of double extensions.\\       
\end{abstract}


\section{Introduction}
In this paper, we consider finite dimensional associative superalgebras over an algebraically closed commutative field $\K$ of characteristic zero. 
A homogeneous symmetric associative superalgebra is an associative superalgebra with non-degenerate, supersymmetric, homogeneous and associative bilinear form. Homogeneous symmetric associative superalgebras, more precisely with even-symmetric structures, appeared in \cite{Im} in order to study symmetric Novikov superalgebras. In particular, the notion of generalized double extension of even-symmetric associative superalgebras was introduced in \cite{Im}. This notion is a generalization in case of associative superalgebras of the notion of double extension of symmetric associative algebras introduced in \cite{An}. A different  construction, namely $T^*-$extension, was given in \cite{Bo} to describe nilpotent associative algebras. By using the notion of double extension introduced in \cite{An} and the notion of $T^*-$extension, the descriptions of symmetric associative commutative algebras was obtained in \cite{Am}.\\  

 Finite dimensional simple associative superalgebras over a field $\K$ of characteristic $\neq 2$ was classified in   \cite{Wl}. In the case when $\K$ is an algebraically closed field, one can find the list of finite dimensional simple associative superalgebras  in \cite{Al}. Contrary to what happens in case of Lie superalgebras, in the third section of this paper, we prove that all simple associative superalgebras admit homogeneous symmetric structures.
 More precisely, we give the homogeneous symmetric structure on every simple associative superalgebra.  Next, we prove that if $A$ is an associative superalgebra with non null product, then it can not admit simultaneously even-symmetric and odd-symmetric structure. The fourth section will be devoted to the descriptions of associative superalgebras $A=A_{\bar{0}}\oplus A_{\bar{1}}$ with homogeneous symmetric structures such that $A_{\bar{0}}$ is a semi-simple $A_{\bar{0}}$-bimodule. In particular, in the  case of odd-symmetric associative superalgebras whose even parts are semi-simple bimodules, we give an interesting description without using the notions of double extensions. Finally, in the last section, we recall the notion of the generalized double extension of even-symmetric associative superalgebras introduced in \cite{Im} and we introduce the generalized double extension of odd-symmetric associative superalgebras in order to give inductive descriptions of these two types of associative superalgebras.\\

\section{Definitions and preliminaries}
A superalgebra $A$ is a $\Z_{2}$-graded algebra $A= A_{\bar{0}}\oplus A_{\bar{1}}$ over $\K$ (i.e $A_{\alpha}.A_{\beta}\subseteq A_{\alpha + \beta}$ for $\alpha$, $\beta \in \Z_{2}$). An element $x$ in $A_{\mid x\mid}$, where $\mid x\mid = \bar{0}, \bar{1}$, is said to be homogeneous of degree $\mid x\mid$. An associative superalgebra is just a superalgebra that is associative as an ordinary algebra. Let $Bil(A,\K)$ be the set of all bilinear forms on $A$. Following \cite{Sc}, $Bil(A,\K)$ is a $\Z_{2}$-graded vector space such that: 
$${Bil(A,\K)}_{\gamma}=\left\{B\in Bil(A,\K)\,; B(A_{\alpha},A_{\beta}) \subset {\K}_{\alpha+\beta+\gamma}; \alpha ,\beta\in \Z_{2}\right\}, \forall \gamma\in \Z_{2},$$
where the $\Z_{2}$ graduation of the field $\K$ is given by: $(\K)_{\bar{0}}=\K$ et $(\K)_{\bar{1}}=\left\{0\right\}$.
An element $B$ in ${Bil(A,\K)}_{\gamma}$, where $\gamma = \bar{0},\bar{1}$, is said to be homogeneous of degree $\gamma$. More precisely, if $B\in {Bil(A,\K)}_{\bar{0}}$ (resp. $B\in {Bil(A,\K)}_{\bar{1}}$) then $B$ is called even (resp. odd) bilinear form.\\ 

\begin{defi}
Let $(A,.)$ be an associative superalgebra. A homogeneous bilinear form $B$ on $A$ is 
\begin{enumerate}
\item[(i)] supersymmetric if: $B(x,y)=(-1)^{\mid x\mid \mid y\mid}B(y,x),\, \forall x\in A_{\mid x\mid}, \, y\in A_{\mid y\mid}$;
\item[(ii)] associative if: $B(x.y,z)=B(x,y.z),\, \forall x,y,z\in A$; 
\item[(iii)] non-degenerate if: $x\in A$ satisfies $B(x,y)=0,\ \forall y\in A$, then $x=0$. 
\end{enumerate}
\end{defi}
\begin{defi}
An even-symmetric (resp. odd-symmetric) associative superalgebra $A$ is an associative superalgebra provided with an even (resp. odd), supersymmetric, associative and non-degenerate bilinear form $B$. $A$ provided with $B$ is denoted by $(A,B)$ and $B$ is called an even-symmetric (resp. odd-symmetric) structure on $A$.\\
\end{defi}
 
 In this paper, we study associative superalgebras with homogeneous symmetric structures. A natural question that arises is: Which condition(s) must an associative superalgebra satisfy so that it does not admits simultaneously even-symmetric and odd-symmetric structure? The answer of this question is in the following subsection which we begin by characterizing  even-symmetric associative superalgebras and odd-symmetric associative superalgebras.\\

\subsection{Incompatibility of even-symmetric and odd-symmetric structures}\label{incompatibility}

In the following, we consider that $A$ is an associative superalgebra and we recall that for an associative superalgebra  $A= A_{\bar{0}}\oplus A_{\bar{1}}$, we have:\\
 
$\bullet$ $A_{\bar{0}}$ and $A_{\bar{1}}$ possess a structure of $A_{\bar{0}}$-bimodules by means of  $(L,R)$ where $\forall x\in A_{\bar{0}}$, $L(x):= L_x$ and $R(x):=R_x$ are respectively the left and right multiplications of $A_i$ by $x$, where $i\in\left\{\bar{0},\bar{1}\right\}$.\\

$\bullet$ ${A_{\bar{0}}}^*$ and ${A_{\bar{1}}}^*$ possess a structure of $A_{\bar{0}}$-bimodules by means of  $(L^*,R^*)$ such that for $ x\in A_{\bar{0}},$ $$L^*(x)(f):= f\circ R_x\ \ \mbox{and} \ \ R^*(x)(f):= f\circ L_x,\ \ \forall f\in {A_i}^* \,\mbox{where}\, i\in \left\{\bar{0}, \bar{1}\right\}.$$ 
  
\begin{proposition}\label{caracterisationpaire}
 $A$ is an even-symmetric associative superalgebra if and only if there exist two isomorphisms of $A_{\bar{0}}-$bimodules $\phi_{\bar{0}}: A_{\bar{0}}\longrightarrow {A_{\bar{0}}}^{*}$ and $\phi_{\bar{1}}: A_{\bar{1}}\longrightarrow {A_{\bar{1}}}^{*}$ such that:
\begin{eqnarray*}
\phi_i(x)(y)&=&{(-1)}^i \phi_i(y)(x), \, \forall x,y\in A_i,\, \forall i\in\left\{\bar{0},\bar{1}\right\}\\
\phi_{1}(x.y)(z)&=&\phi_{0}(x)(y.z),\ \forall\, x\in A_{\bar{0}},\, y,z\in A_{\bar{1}}
\end{eqnarray*}
\end{proposition}

\begin{proof}
Let $B$ be an even-symmetric structure on $A$. It is clear that $B_{\bar{0}}:= {B\mid}_{A_{\bar{0}}\times A_{\bar{0}}}$ (resp. $B_{\bar{1}}:= {B\mid}_{A_{\bar{1}}\times A_{\bar{1}}}$) is symmetric (resp. antisymmetric), non-degenerate bilinear form over  $A_{\bar{0}}$ (resp. $A_{\bar{1}}$). For  $i\in\left\{\bar{0}, \bar{1}\right\}$, we consider the following map $\phi_{i}: A_{i}\longrightarrow {A_{i}}^{*}$ defined by $\phi_{i}(x):= B_i(x,.)$, $\forall\, x\in A_{i}$. We can easily check that $\forall\, i\in\left\{\bar{0}, \bar{1}\right\}$, $\phi_{i}$ is an isomorphism of vector spaces such that: $\phi_{1}(x.y)(z)=\phi_{0}(x)(y.z),\ \forall\, x\in A_{\bar{0}},\, y,z\in A_{\bar{1}}$ and $\phi_i(x)(y)={(-1)}^i \phi_i(y)(x), \, \forall x,y\in A_i$. Finally to prove that $\phi_{i}$ is a morphism of $A_{\bar{0}}-$bimodules, let us consider $y\in A_{\bar{0}}$, $x,z\in A_{i}$,
$$\phi_{i}(R_y(x))(z)=\phi_{i}(x.y)(z)=B_i(x.y,z)=(B_i(x,.)\circ L_y)(z)=R^*(y)(\phi_{i}(x))(z),$$
$$\phi_{i}(L_y(x))(z)=\phi_{i}(y.x)(z)=B_i(y.x,z)=B_i(x,z.y)= (\phi_{i}(x)\circ R_y)(z)= L^*(y)(\phi_{i}(x))(z).$$ 
Consequently, $\forall i\in \left\{\bar{0},\bar{1}\right\}$, $\phi_{i}$ is an isomorphism of $A_{\bar{0}}-$bimodules. Conversely, suppose that for $i\in \left\{\bar{0},\bar{1}\right\}$, there exist an isomorphism of $A_{\bar{0}}-$bimodules $\phi_{i}: A_{i}\longrightarrow A_{i}^*$ such that $\phi_i(x)(y)={(-1)}^i \phi_i(y)(x), \, \forall x,y\in A_i$ and  
$\phi_{1}(x.y)(z)=\phi_{0}(x)(y.z),\ \forall\, x\in A_{\bar{0}},\, y,z\in A_{\bar{1}}$. We define the following bilinear form on $A$ by: $B(x,y):=\phi_{i}(x)(y)$,$\forall x,y\in A_i$ and $B(A_{\bar{0}},A_{\bar{1}})=\left\{0\right\}$. By definition, $B$ is even, non-degenerate and supersymmetric such that  for $x,y,z\in A_{\bar{0}}$, $$B(x.y,z)=\phi_0(R_y(x))(z)=(\phi_{0}(x)\circ L_y)(z)=B(x,y.z).$$ 
And for $x,z\in A_{\bar{1}}$, $y\in A_{\bar{0}}$, we have $$B(x.y,z)=\phi_1(R_y(x))(z)=(\phi_1(x)\circ L_y)(z)=B(x,y.z).$$ 
Hence, we conclude that $(A,B)$ is an even-symmetric associative superalgebra.\\
\end{proof}

\begin{proposition}\label{caracterisationimpaire}
 $A$ is an odd-symmetric associative superalgebra if and only if there exist
an isomorphism of $A_{\bar{0}}-$bimodules $\phi: A_{\bar{1}}\longrightarrow {A_{\bar{0}}}^{*}$ such that $\phi(x)(y.z)= \phi(z)(x.y)$, $\forall x, y, z\in A_{\bar{1}}$.
\end{proposition} 

\begin{proof}
 Let $B$ be an odd-symmetric structure on $A$. We consider the following linear map $\phi: A_{\bar{1}}\longrightarrow A_{{\bar{0}}}^*$ defined by $\phi(x):= B(x,.)$, $\forall x\in A_{\bar{1}}$. As $B$ is odd and non-degenerate bilinear form, then $\phi$ is an isomorphism of vector spaces. In addition, by the associativity of $B$, it is clear that $\phi$ satisfies $\phi(x)(y.z)= \phi(z)(x.y)$, $\forall x, y, z\in A_{\bar{1}}$. Finally, to prove that $\phi$ is a morphism of  $A_{\bar{0}}$-bimodules, let us consider $y,z \in A_{\bar{0}}$ and $x\in A_{\bar{1}}$, then  
$$\phi(R_{y}(x))(z)= B(x.y,z)= B(x,y.z)= B(x, L_{y}(z))= (B(x,.)\circ L_{y})(z)= R^*(y)(\phi(x))(z)$$
$$\phi(L_{y}(x))(z)= B(y.x,z)= B(x,z.y)= B(x, R_{y}(z))= (B(x,.)\circ R_{y})(z)= L^*(y)(\phi(x))(z).$$
Consequently, $\phi$ is an isomorphism of $A_{\bar{0}}-$bimodules. Conversely, let $\phi: A_{\bar{1}}\longrightarrow A_{{\bar{0}}}^*$ be an isomorphism of $A_{\bar{0}}-$bimodules such that $\phi(x)(y.z)= \phi(z)(x.y)$, $\forall x,y,z\in A_{\bar{1}}$. We consider the following bilinear form defined on $A$ by: $B(A_{\bar{0}}, A_{\bar{0}})= B(A_{\bar{1}}, A_{\bar{1}})= \left\{0\right\}$ and  $B(x,y)=B(y,x)= \phi(x)(y)$, $\forall x\in A_{\bar{1}}, y\in A_{\bar{0}}$. By definition, $B$ is odd and non-degenerate. In addition for $x\in A_{\bar{1}}, y,z \in A_{\bar{0}}$ we have: 
$$B(x.y,z)= \phi(x.y)(z)= \phi(R_y(x))(z)= (\phi(x)\circ L_y)(z)= \phi(x)(y.z)=B(x,y.z),$$
$$B(z, y.x)= B(y.x,z)= \phi(L_y(x))(z)= (\phi(x)\circ R_y)(z)= B(x, z.y)= B(z.y,x).$$
Then, we conclude that $(A,B)$ is an odd-symmetric associative superalgebra.\\
\end{proof}

By Proposition \ref{caracterisationpaire} and Proposition \ref{caracterisationimpaire}, we have the following proposition that we find which condition(s) must an associative superalgebra satisfy so that it does not possess simultaneously  an even-symmetric and an odd-symmetric structure.\\

\begin{proposition}\label{incompatible}
Any associative superalgebra with non-null product does not admits simultaneously an even-symmetric and an odd-symmetric structure.
\end{proposition}

\begin{proof}
Supposing that $A$ is endowed with an even-symmetric and an odd-symmetric structure which are denoted respectively by $B$ and $\bar{B}$ and proving that $A$ is with null product. Following Proposition \ref{caracterisationpaire} (resp. Proposition \ref{caracterisationimpaire}), we have $A_{\bar{0}}$ and ${A_{\bar{0}}}^*$ (resp. $A_{\bar{1}}$ and ${A_{\bar{0}}}^*$ ) are isomorphic as $A_{\bar{0}}$-bimodules by means of $\phi_0$ (resp. $\phi$). Consequently, we have $A_{\bar{0}}$ and $A_{\bar{1}}$ are isomorphic as $A_{\bar{0}}$-bimodules by means of $\psi= \phi^-\circ \phi_0$, where $\phi^-$ is the bijective map of $\phi$. So, we can define the following bilinear form $f$ on $A$ as follows : $f(x,y)=B_{\bar{1}}(\psi(x), \psi(y)), \, \forall x,y\in A_{\bar{0}}$, where $B_{\bar{1}}= {B\mid}_{A_{\bar{1}}\times A_{\bar{1}}}$. The fact that $B_{\bar{1}}$ is an antisymmetric non-degenerate bilinear form on $A_{\bar{1}}$ and $\psi$ is an isomorphism implies that $f$ is an antisymmetric non-degenerate bilinear form on $A_{\bar{0}}$. Moreover, $f$ is associative since for $x,y,z\in A_{\bar{0}}$, we have  
 \begin{eqnarray*} f(x.y,z)&=&B_{\bar{1}}(\psi(x.y), \psi(z))=B_{\bar{1}}(\psi(R_y(x)), \psi(z))=B_{\bar{1}}(R_y(\psi(x)), \psi(z))\\ 
&=&B_{\bar{1}}(\psi(x), L_y(\psi(z)))=B_{\bar{1}}(\psi(x), \psi(L_y(z)))=B_{\bar{1}}(\psi(x), \psi(y.z))=f(x,y.z). 
\end{eqnarray*}

Using the antisymmetry and the associativity of $f$, we deduce that ${A_{\bar{0}}}^2\subseteq Ker(f)$. In addition, by the non-degeneracy of $f$ we obtain ${A_{\bar{0}}}^2=\left\{0\right\}$. Finally, by a simple computation where we use  the associativity of $\bar{B}$ and $B$, we show easily that $A_{\bar{0}}A_{\bar{1}}=\left\{0\right\}=A_{\bar{1}}A_{\bar{0}}$ and  ${A_{\bar{1}}}^2=\left\{0\right\}$.\\
\end{proof}

In the sequel, we consider associative superalgebras with non-null product unless otherwise stated. Meaning that we don't  consider associative superalgebras with simultaneously even-symmetric and odd-symmetric structures.\\

\begin{remark}\label{rmqincompatible}
Recall that in \cite{Sc}, classical Lie superalgebras don't have simultaneously quadratic and odd-quadratic structure. Later, in \cite{He}, it has been proved that any perfect Lie superalgebra with even part is a reductive Lie algebra does not possess simultaneously  quadratic and odd-quadratic structure. More generally, we can see in the same way of the previous proposition , that any non abelian Lie superalgebra does not possess simultaneously a quadratic  and an odd-quadratic structures.\\
\end{remark}

\subsection{Some definitions and results independent of the parity of symmetric structure}

\begin{defi}
Let $(A,B)$ be an associative superalgebra with homogeneous symmetric structure and $I$ a graded two-sided ideal of $A$.
\begin{enumerate}
\item[(i)] $I$ is called non-degenerate (resp. degenerate) if the restriction of $B$ to $I\times I$ is a non-degenerate  (resp. degenerate) bilinear form;
\item[(ii)] $(A,B)$ is called $B$-irreducible if $A$ does not contains non-degenerate graded two-sided ideal other than $\left\{0\right\}$ and $A$; 
\item[(iii)] $I$ is called $B$-irreducible if $I$ is non-degenerate and it does not contains non-degenerate graded two-sided ideal of $A$ other than $\left\{0\right\}$ and $I$;   
\item[(iv)] $I$ is called totally isotropic if $B(I,I)=\left\{0\right\}$.\\
\end{enumerate} 
\end{defi}

\begin{lemma}\label{Iideal}
Let $(A,B)$ be an associative superalgebra with homogeneous symmetric structure, $I$ a graded two-sided ideal of $A$ and $J$ its orthogonal with respect to $B$. Then, 
\begin{enumerate}
\item[(i)] $J$ is a graded two-sided ideal of $A$ and $I.J= J.I= \left\{0\right\}$,
\item[(ii)] if  $I$ is non-degenerate, then $A = I\oplus J$ and  $J$ is also non-degenerate, 
\item[(iii)] if $I$ is semi-simple, then $I$ is non-degenerate.
\end{enumerate}
\end{lemma}

\begin{proof}
The proof of assertion $(i)$ is similar to the proof of Lemma III.5 in \cite{Im} in case of even-symmetric associative superalgebras. The assertion $(ii)$ is clear. $(iii)$ By the associativity of $B$, we have $B((I\cap J).(I\cap J),A)=B(I\cap J, (I\cap J).A)=\left\{0\right\}$ and which imply that $I\cap J=\left\{0\right\}$ since $I$ is semi-simple. Consequently, we have $A= I\oplus J$ and so by the non-degeneracy of $B$, we conclude that $I$ is non-degenerate.\\
\end{proof}

The following proposition is an immediate consequence of Lemma \ref{Iideal}.\\
 
\begin{proposition}\label{depimp}
Every associative superalgebra with homogeneous symmetric structure is an orthogonal direct sums of $B$-irreducible graded two-sided ideals.\\
\end{proposition}

Following the proposition above, the description of even-symmetric or odd-symmetric associative superalgebras amounts to the description of those that are $B$-irreducible. Among even-symmetric or odd-symmetric $B$-irreducible associative  superalgebras, we find those that are $B$-irreducible non-simple. In \cite{Im}, 
we introduced the notion of the generalized double extension of even-symmetric $B$-irreducible non-simple associative superalgebras to describe symmetric Novikov superalgebras. The minimal graded two-sided ideal played an important role for the introduction of this notion. In the following, we will proceed similar to \cite{Im} to give an inductive description of homogeneous-symmetric associative superalgebras. For this reason, we start by giving a characterization of minimal graded two-sided ideal of associative superalgebras with homogeneous structures.\\ 

\begin{defi}
Let $A$ be an associative superalgebra and $I$ a graded two-sided ideal of $A$. $I$ is called minimal if $I\notin \left\{\left\{0\right\}, A\right\}$ and if $J$ is a graded two-sided ideal of $A$ such that $J\subseteq I$ then $J\in \left\{\left\{0\right\}, I\right\}$.\\
\end{defi}

\begin{lemma}\label{minimal}
Let $(A,B)$ be an associative superalgebra with homogeneous symmetric structure and $I$ a minimal graded two-sided ideal of $A$, then $I$ is simple or $I$ is with null product such that $AI=I=IA$ or $I= \K x $ where $x$ is a homogeneous element of $Ann(A)$.\\
\end{lemma}

\begin{proof}
Let $I$ be a minimal graded two-sided ideal of $A$ and $J$ its orthogonal with respect to $B$. As $I\cap J$ is a graded two-sided ideal of $A$ which is included in $I$, then  we have $I\cap J=\left\{0\right\}$ or $I\cap J= I$. First, we suppose that $I\cap J=\left\{0\right\}$, then $A=I\oplus J$ and it follows that any graded two-sided ideal of $I$ is a graded two-sided ideal of $A$. So, we deduce that $I$ is simple. Next, if  $I\cap J=I$, then by Lemma \ref{Iideal}, $I$ is with null product. On the other hand, $AI$ and $IA$ are two graded two-sided ideals of $A$ which are included in $I$. Consequently, by the minimality of $I$, we have  $AI\in \left\{\left\{0\right\},I\right\}$ and $IA\in \left\{\left\{0\right\},I\right\}$. We can deduce by a simple computation where we use the associativity and the non-degeneracy of $B$, that we have uniquely $AI=I=IA$. It follows that $I$ is a graded two-sided ideal with null product such that $AI=I=IA$ or $I=\K x$, where  $x$ is a homogeneous element of $ Ann(A)$.\\
\end{proof}

\begin{remark}\label{minimalirr} 
If $(A,B)$ is a non-simple associative superalgebra with $B$-irreducible homogeneous symmetric structure such that $dim A>1$, then any minimal graded two-sided ideal of $A$ is totally isotropic and so it is not simple.\\  
\end{remark}
  
\begin{lemma}\label{quotient}
Let $(A,B)$ be a non-simple associative superalgebra with $B$-irreducible homogeneous symmetric structure such that $dim A>1$. Let $I$ be a minimal graded two-sided ideal of $A$ and $J$ its orthogonal with respect to $B$. Then:  
\begin{enumerate}
\item[(i)] $A/J$ is a simple associative superalgebra or a one-dimensional superalgebra with null product,
\item[(ii)] $A/J$ is a one-dimensional superalgebra with null product if and only if $I=\K x$, where $x$ is a homogeneous element of $Ann(A)$.\\
\end{enumerate}
\end{lemma} 

\begin{proof}
$(i)$ As $I$ is a minimal graded two-sided ideal of $A$, then its orthogonal $J$ is a maximal two-sided ideal of $A$ and consequently the set of all ideal of $A/J$ is composed by $\left\{\left\{0\right\}, A/J\right\}$. So, we have $A/J$ is a simple associative superalgebra or a one-dimensional superalgebra with null product. $(ii)$ We suppose that $A/J$ is the one-dimensional superalgebra with null product. Then, $dim I=1$ and which implies by hypothesis that $I$ is totally isotropic. So $I=\K x$ where $x$ is a homogeneous element of $J$. As $I\subseteq J$ and $B$ is non-degenerate, there exist a homogeneous element $y$ in $A$ such that $A= J\oplus \K y$ with $B(x,y)\neq 0$ and $y^2\in J$ since $A/J$ is with null product. Now, we prove that $x\in Ann(A)$. It is clear, by Lemma \ref{Iideal}, that $x.J= \left\{0\right\}= J.x$. On the over hand, $B(x.y,A)=B(x.y, J\oplus \K y)= \left\{0\right\}$. Hence $x.y=y.x=0$ and consequently $x\in Ann(A)$. Conversely, let us consider that $I=\K x$, where $x\in Ann(A)$, and $J$ its orthogonal with respect to $B$. According to Remark \ref{minimalirr}, we have $I$ is totally isotropic. So since $B$ is non-degenerate, there exist a homogeneous element $y$ in $A$ such that $A=J\oplus \K y$ and $B(x,y)\neq 0$. Proving that $y^2\in J$. Supposing that  $y^2= j+ky$, where $j\in J$ et $k\in \K$. Then, $0=B(x.y,y)= B(x,y^2)= k B(x,y)$ and consequently, we obtain that $k=0$. So $y^2\in J$ and we deduce that $A/J$ is a one-dimensional superalgebra with null product.\\ 
\end{proof}

\section{ Simple associative superalgebras with homogeneous symmetric structures}

Following Proposition \ref{incompatible}, simple associative superalgebras don't admit simultaneously even-symmetric and odd-symmetric structures. In \cite{Al}, We find the list of all simple associative superalgebras over an algebraically closed field $\K$ up isomorphism. In this section, we will prove that all simple associative superalgebras admit homogeneous symmetric structures. We give explicitly simple associatives superalgebras which admit even-symmetric structures and those that admit odd-symmetric structures. The list of all simple associative superalgebras given in \cite{Al} is as follows:\\ 

 (a) $A= M_{n}(\K)$, The algebra of square matrix of order $n$ over $\K$ such that: 

$A_{\bar{0}}= \left\{\left(
\begin{array}{cc}
a & 0\\
0 & b
\end{array}
\right): a\in M_{r}(\mathbb K), b\in M_{s}(\K) \right\}$, $A_{\bar{1}}= \left\{ \left(
\begin{array}{cc}
0 & c\\
d & 0
\end{array}
\right): c\in M_{r\times s}(\mathbb K), d\in M_{s\times r}(\K) \right\}$, where $r\geq 1$, $s\geq 0$ and $r+s=n$. This superalgebra is denoted by $M_{r,s}(\K)$. \\

(b) The sub-superalgebra $A= A_{\bar{0}}\oplus A_{\bar{1}}$ of $M_n(\mathbb K)$ such that:

$A_{\bar{0}}= \left\{ \tilde{a}:=\left(
\begin{array}{cc}
a & 0\\
0 & a
\end{array}
\right): a\in M_{n}(\mathbb K)    \right\}$, $A_{\bar{1}}= \left\{ \bar{b}:=\left(
\begin{array}{cc}
0 & b\\
b & 0
\end{array}
\right): b\in M_{n}(\mathbb K)   \right\}$. 

We denote this superalgebra by $Q_{n}(\mathbb K).$\\

In Proposition \ref{simplepaire} (resp. Proposition \ref{simpleimpaire}) below, we prove that $M_{r,s}(\K)$, where $r\geq 1$ and $s\geq 0$, are the only even-symmetric simple associative superalgebras (resp. $Q_{n}(\K)$, where $n\in \N^*$ are the only odd-symmetric simple associative superalgebras).\\

\begin{proposition}\label{simplepaire}
$M_{r,s}(\K)$, where $r\geq 1$ and $s\geq 0$, are the unique even-symmetric simple associative  superalgebras up isomorphism. In particular, the field $\K$ is the unique even-symmetric simple associative super-commutative superalgebra up isomorphism.\\ 
\end{proposition}

\begin{proof}
Let $A$ be a simple associative superalgebra, then by \cite{Al},  $A\in \left\{ M_{r,s}(\K), Q_{n}(\K)\right\}$.\\

\textbf{First case}: Let $A= Q_{n}(\K)$, where $n\in \N^*$. If $n=1$, since $dim (Q_1(\K))=1$, then $Q_1(\K)$ doesn't admits any even-symmetric structure. Suppose that $n>1$ and consider the following  basis of $Q_{n}(\K)$ \\
$\left\{E_{ij}:= \left(
\begin{array}{cc}
e_{ij} & 0\\
0 & e_{ij}
\end{array}
\right),\, F_{ij}:= \left(
\begin{array}{cc}
0 & e_{ij}\\
e_{ij} & 0
\end{array}
\right),\ \mbox{where}\, {(e_{ij})}_{\begin{tabular}{ll}$1\leq i\leq n$,&\\ $1\leq j\leq n$\end{tabular}}\mbox{is the usual basis of}\  M_n(\K) \right\}$. It is clear that
$E_{ij}E_{lk}= {\delta}_{jl}E_{ik}, \ \ F_{ij}F_{lk}= {\delta}_{jl}E_{ik},\ \ E_{ij}F_{lk}= {\delta}_{jl}F_{ik},\ \ F_{ij}E_{lk}= {\delta}_{jl}F_{ik}$.
We suppose that $A$ admits an even-symmetric structure $B$. Then by the associativity of $B$ we have:
\begin{eqnarray*}
B(E_{11},E_{11})&=& B(E_{11}, F_{11}F_{11})=B(E_{11}F_{11},F_{11})= B(F_{11},F_{11})=0,
\end{eqnarray*}
because ${B\vert}_{A_{\bar{1}}\times A_{\bar{1}}}$ is skew symmetric. Now, let $m\in \left\lbrace  2,\cdots,n\right\rbrace $ and $k\in \left\lbrace  1,\cdots,n\right\rbrace$, then 
\begin{eqnarray*}
B(E_{11}, E_{1m})&=&B(E_{11}, E_{11}E_{1m}) =B(E_{11},E_{1m}E_{11})=0.\\
B(E_{11},E_{mk})&=&B(E_{11},E_{m1}E_{1k})=B(E_{11}E_{m1},E_{1k})=0.
\end{eqnarray*}
 So, we obtain that, $E_{11}$ is in the orthogonal of $A$ with respect of $B$ and which contradict the fact that $B$ is even and non-degenerate. We conclude that $Q_n(\K)$ does not admits even-symmetric structure.\\

\textbf{Second case}: We suppose that $A=M_{r,s}(\K)$, where $r\in \N^*$ and $s\in \N$. We consider the following  bilinear form $B$ on $A$ defined by:
\begin{eqnarray*} 
B(\tilde{M},\tilde{M'})&=&tr(aa')-tr(bb'),\ \forall\, \tilde{M}=\left(
\begin{array}{cc}
a & 0\\
0 & b
\end{array}
\right), \,\tilde{M'}=\left(
\begin{array}{cc}
a' & 0\\
0 & b'
\end{array}
\right)\in A_{\bar{0}},\\
B(\bar{N},\bar{N'})&=&tr(cd')-tr(dc'),\ \forall\, \bar{N}=\left(
\begin{array}{cc}
0 & c\\
d & 0
\end{array}
\right), \,\bar{N'}=\left(
\begin{array}{cc}
0 & c'\\
d' & 0
\end{array}
\right)\in A_{\bar{1}},\\
B(A_{\bar{0}},A_{\bar{1}})&=&\left\{0\right\}. 
\end{eqnarray*}
It is clear that $B$ is even super-symmetric such that ${B\mid}_{A_{\bar{0}}\times A_{\bar{0}}}$ is non-degenerate. To show that $B$ is non-degenerate on $A$, it remains to verify that ${B\mid}_{A_{\bar{1}}\times A_{\bar{1}}}$ is non-degenerate. Let $c={(c_{mn})}_{1\leq m\leq r,\\ 1\leq n\leq s}$ such that  $tr(cd)=0$, $\forall d\in M_{s\times r}(\mathbb K)$. Let $i_0\in \left\lbrace 1,\cdots, s\right\rbrace $ and $j_0\in \left\lbrace 1,\cdots,r\right\rbrace $ and consider  $d=e_{i_0j_0}$, where ${(e_{ij})}_{1\leq i\leq s,\\1\leq j\leq r}$ is the canonical basis of $M_{s\times r}(\K)$. Then $c e_{i_0j_0}= {(a_{pq})}_{1\leq p\leq r,\\1\leq q\leq r}$ such that $a_{pp}=0$ if $p\neq j_0$ and $a_{j_0j_0}= c_{j_0i_0}$. Hence, the fact that $tr(ce_{i_0j_0})=0$ implies that $c_{j_0i_0}=0$, $\forall\, 1\leq j_0\leq r,\, 1\leq i_0\leq s$. Consequently $c=0$ and we deduce that ${B\vert}_{A_{\bar{1}}\times A_{\bar{1}}}$ is a non-degenerate bilinear form. In addition, for all $a\in M_{r}(\K)$, $b\in M_{s}(\K)$, $c,c'\in M_{r\times s}(\K)$ and $d,d'\in M_{s\times r}(\K)$, we have
$$B(\left(
\begin{array}{cc}
a & 0\\
0 & b
\end{array}
\right)\left(
\begin{array}{cc}
0 & c\\
d & 0
\end{array}
\right), \left(
\begin{array}{cc}
0 & c'\\
d' & 0
\end{array}
\right))= tr(acd')-tr(bdc')=B(\left(
\begin{array}{cc}
a & 0\\
0 & b
\end{array}
\right), \left(
\begin{array}{cc}
0 & c\\
d & 0
\end{array}
\right)\left(
\begin{array}{cc}
0 & c'\\
d' & 0
\end{array}
\right)) $$
 So, we conclude that $(A,B)$ is an even-symmetric simple associative superalgebra.\\  

\end{proof}

\begin{proposition}\label{simpleimpaire}
$Q_{n}(\K)$, where $n\in \N^*$, are the unique odd-symmetric simple associative  superalgebras up isomorphism. In particular, $Q_{1}(\K)$ is the unique odd-symmetric simple associative super-commutative superalgebra up isomorphism.\\
\end{proposition}

\begin{proof}
Let $A$ be a simple associative superalgebra, then we have  $A\in \left\{ M_{r,s}(\K), Q_{n}(\K)\right\}$.\\

\textbf{First case}: Let $A= M_{r,s}(\K)$, where $r\geq 1$ and $s\geq 0$. If we suppose that $B$ is an odd-symmetric structure on $A$, then $dim A_{\bar{0}}= dim A_{\bar{1}}$ and so we have $r=s$. Set $n=r=s$ and let 
$\left(
\begin{array}{cc}
0 & c\\
d & 0
\end{array}
\right)\in A_{\bar{1}}$, where $c,d\in M_{n}(\K)$, then by the associativity of $B$, we have: 
\begin{eqnarray}
B(\left(
\begin{array}{cc}
0 & c\\
d & 0
\end{array}
\right), \left(
\begin{array}{cc}
I_{n} & 0\\
0 & 0
\end{array}
\right))&=& 
B(\left(
\begin{array}{cc}
0 & c\\
0 & 0
\end{array}
\right)\left(
\begin{array}{cc}
0 & 0\\
0 & I_{n}
\end{array}
\right),\left(
\begin{array}{cc}
I_{n} & 0\\
0 & 0
\end{array}
\right))\nn \\  &+& 
 B(\left(
\begin{array}{cc}
0 & 0\\
0 & I_{n}
\end{array}
\right)\left(
\begin{array}{cc}
0 & 0\\
d & 0
\end{array}
\right), 
\left(
\begin{array}{cc}
I_{n} & 0\\
0 & 0
\end{array}
\right))\nn \\ &=& 0.\nn
\end{eqnarray} Consequently, $B(\left(\begin{array}{cc}
I_{n} & 0\\
0 & 0
\end{array}
\right), A_{\bar{1}} )=\left\{0\right\}$.
 On the other hand, as $B$ is odd, we obtain that   $B(\left(\begin{array}{cc}
I_{n} & 0\\
0 & 0
\end{array}
\right), A_{\bar{0}})=\left\{0\right\}$. So $\left(
\begin{array}{cc}
I_{n} & 0\\
0 & 0
\end{array}
\right) $ is in the orthogonal of $M_{n,n}(\K)$ with respect of $B$. Consequently, we conclude that $M_{n,n}(\K)$ are not odd-symmetric associative superalgebras.\\

\textbf{Second case}: In this case, we consider $A= Q_{n}(\K)$, where $n\in \N^*$, and $B$ is a bilinear form on $A$ defined by:  $B(\tilde{a},\bar{b}):= tr(ab)$ and  $B$ is null elsewhere. It is clear that $B$ is odd, super-symmetric and non-degenerate bilinear form. In addition, for $a,a',b,b',b''\in M_{n}(\K)$, we have: 
$$B(\tilde{a}\tilde{a'}, \bar{b})= B(\widetilde{aa'}, \bar{b})= tr((aa')b)= tr(a(a'b))= B(\tilde{a}, \overline{a'b})= B(\tilde{a}, \tilde{a'}\bar{b}),$$
$$B(\bar{b}\bar{b'},\bar{b''})= B(\widetilde{bb'},\bar{b''})= tr((bb')b'')= tr(b(b'b''))= B(\bar{b},\widetilde{b'b''})= B(\bar{b},\bar{b'}\bar{b''}).$$
Hence $(Q_{n}(\K),B)$ is odd-symmetric simple associative superalgebra.\\

\end{proof}

\begin{remark}
 Let $(A,B)$ be an even-symmetric (resp. odd-symmetric) simple associative superalgebra. If $B'$ is an other even-symmetric structure (resp. odd-symmetric structure) on $A$, then there exist $\lambda\in \K$ such that $B'=\lambda B$. Indeed, as $B$ and $B'$ are two bilinear forms non-degenerate of same degree, then there exist an even automorphism 
 $\varphi$ of $A$ such that $B'(x,y)=B(\varphi(x),y)$, $\forall x,y\in A$. Using the associativity of $B$ and those of $B'$, we obtain for all $x,y,z\in A$ 
 \begin{eqnarray*}
B(\varphi(x).y,z)&=&B(\varphi(x),y.z)= B'(x,y.z)=B'(x.y,z)=B(\varphi(x.y),z)\\ 
B(x.\varphi(y),z)&=& {(-1)}^{\mid z\mid(\mid x\mid + \mid y\mid)}B(z.x,\varphi(y)) = {(-1)}^{\mid x\mid(\mid y\mid + \mid z\mid)}B(\varphi(y), z.x)\\
 &=& {(-1)}^{\mid x\mid(\mid y\mid + \mid z\mid)} B'(y, z.x)= {(-1)}^{\mid z\mid(\mid x\mid + \mid y\mid)}B'(z.x, y)= B'(x.y, z)= B(\varphi(x.y), z).    
\end{eqnarray*}

Now using the non-degeneracy of $B$, we obtain that $\varphi$ is a morphism of $A$-bimodules. Hence, we can show easily that $Ker(\varphi-\lambda id)$ is a non-null graded two-sided ideal of $A$. So, we conclude by the simplicity of $A$ that $A=Ker(\varphi-\lambda id)$ and consequently $\varphi=\lambda id$.\\  
 \end{remark}

\section{ Associative superalgebras with homogeneous symmetric structures whose even parts are semi-simple bimodules}

This section is composed essentially of two parts. The first one is dedicated to the study of odd-symmetric associative superalgebras whose even parts are semi-simple bimodules and the second one is reserved to the study of even-symmetric associative superalgebras whose even parts are semi-simple bimodules.

\subsection{Odd-symmetric associative superalgebras whose even parts are semi-simple bimodules} 
 
We begin with the following two examples which will be useful in the sequel.\\

 \textbf{Example 1:} 
Let $(A,.)$ be an associative algebra. We consider the $\Z_{2}$-graded vector space $P(A^*)$ such that  $(P(A^*))_{\bar{0}}=\left\{0\right\}$ and $(P(A^*))_{\bar{1}}=A^*$, where $A^*$ is the dual space of $A$. In the $\Z_{2}$-graded vector space $A\oplus P(A^*)$, we define the product $\star$ and  the bilinear form $B$ respectively by:
 
\begin{eqnarray}
 (x+f)\star(y+h)&=&x.y + L^*(x)(h)+ R^*(y)(f),\ \forall x,y\in A, \ \ f,h\in P(A^*),\\
 B(f,x) &=& B(x,f)= f(x)\ \ \mbox{and}\ \  B(A,A)= \left\{0\right\}= B(P(A^*),P(A^*)).
 \end{eqnarray}
 $(A\oplus P(A^*),\star,B)$ is an odd-symmetric associative superalgebra that we call semi-direct product of $A$ by $P(A^*)$ by means of $(L^*,R^*)$, where $(L^*,R^*)$ is defined in subsection \ref{incompatibility}.\\

 \textbf{Example 2:}
 Let $A=A_{\bar{1}}$ be an associative superalgebra, $P(A^*)$ a $\Z_{2}$-graded vector space such that ${(P(A^*))}_{\bar{0}}= A^*, {(P(A^*))}_{\bar{1}}=\left\{0\right\}$ and $\gamma: A\times A \longrightarrow P(A^*)$ a bilinear map which satisfies $\gamma(x,y)(z)=\gamma(y,z)(x),\ \forall x,y,z\in A$. In the $\Z_{2}$-graded vector space  $P(A^*)\oplus A$, we define the following product $\star$ and the bilinear form $B$ respectively by: 
 \begin{eqnarray}
 (f+x)\star(g+y)&=& \gamma(x,y),\ \forall x,y\in A, \ \  f,h\in (P(A^*),\\
B(f,x)&=& B(x,f) = f(x)\ \ \mbox{and}\ \  B(A,A)= \left\{0\right\}= B(P(A^*),P(A^*)).
\end{eqnarray}  
 $(P(A^*)\oplus A,\star,B)$ is an odd-symmetric associative superalgebra that we call generalized semi-direct product of $P(A^*)$ by $A$ by means of $\gamma$.\\

In order to give an inductive description of odd-symmetric associative superalgebras whose even parts are semi-simple bimodules, we begin by studying the particular case of these superalgebras. More precisely, we start by giving an inductive description of odd-symmetric associative superalgebras whose even parts are semi-simple associative algebras.\\

\begin{lemma}\label{annnul}
Let $(A,B)$ be an odd-symmetric associative superalgebra such that $A_{\bar{0}}$ is a semi-simple associative algebra, then $Ann(A)=\left\{0\right\}$.\\
\end{lemma}
 
\begin{proof}
 It is clear that $Ann(A)\cap A_{\bar{0}}\subseteq Ann(A_{\bar{0}})$. Consequently $Ann(A)\cap A_{\bar{0}}=\left\{0\right\}$ because $A_{\bar{0}}$ is a semi-simple algebra. On the other hand, for $x\in Ann(A)\cap A_{\bar{1}}$ we have  $B(x,A)= B(x,A_{\bar{0}})= B(xA_{\bar{0}},A_{\bar{0}})=\left\{0\right\}$. So we obtain, by the non-degeneracy of $B$, that $x=0$ and we conclude that $Ann(A)= \left\{0\right\}$.\\ 
\end{proof} 

\begin{proposition}\label{min}
 Let $(A,B)$ be an odd-symmetric $B$-irreducible non-simple associative superalgebra such that $A_{\bar{0}}$ is a semi-simple associative algebra. Then $I$ is a minimal graded two-sided ideal of $A$ if and only if $I$ is a non-trivial irreducible $A_{\bar{0}}$-sub-bimodule of $A_{\bar{1}}$ such that $I. A_{\bar{0}}= I = A_{\bar{0}}.I$\\  
 \end{proposition}

\begin{proof}
  Let $I$ be a minimal graded two-sided ideal of $A$, then according Remark \ref{minimalirr}, $I$ is totally isotropic. Consequently, by Lemma \ref{Iideal}, $I$ is with null product. As $A_{\bar{0}}$ is a semi-simple associative algebra, it follows that $I_{\bar{0}}=\left\{0\right\}$. So, we have  $I\subseteq A_{\bar{1}}$ and $I. A_{\bar{1}}= A_{\bar{1}}. I= \left\{0\right\}$. Using the fact that $I. A_{\bar{1}}= A_{\bar{1}}. I= \left\{0\right\}$, we can see easily that $I. A_{\bar{0}}$ and $A_{\bar{0}} .I$ are two two-sided ideals of $A$ which are included in $I$. Hence, by the minimality of $I$, we have $I.A_{\bar{0}}\in \left\{\left\{0\right\}, I\right\}$ and  $A_{\bar{0}}.I\in \left\{\left\{0\right\}, I\right\}$. By a simple computation where we use Lemma \ref{annnul} and  the associativity of $B$, we deduce that we have $A_{\bar{0}}.I=I=I.A_{\bar{0}}$. Finally, to prove that $I$ is irreducible, it is just sufficient to verify that all $A_{\bar{0}}$-sub-bimodule of $I$ are  graded ideals of $A$ which are included in $I$ and then we conclude by the minimality of $I$. Conversely, we suppose that $I$ is a  non-trivial irreducible $A_{\bar{0}}$-sub-bimodule of $A_{\bar{1}}$ such that $I .A_{\bar{0}}= A_{\bar{0}}. I= I$. Let $J$ be a graded two-sided ideal of $A$ such that $J\subseteq I$. It is clear that $J$ is a $A_{\bar{0}}$-sub-bimodule of $A_{\bar{1}}$ which is included in $I$. Since $I$ is an irreducible $A_{\bar{0}}$-sub-bimodule, then we have $J=\left\{0\right\}$ or $J= I$ and so it follows that $I$ is minimal.\\     
 \end{proof}

\begin{lemma}\label{SS'}
Let $(A=A_{\bar{0}}\oplus A_{\bar{1}},B)$ be an odd-symmetric $B$-irreducible non-simple associative superalgebra such that $A_{\bar{0}}$ is a semi-simple associative algebra and $I$ a minimal graded two-sided ideal of $A$. If we note $S'= A_{\bar{0}}\cap J$, where $J$ is the orthogonal of $I$ with respect to $B$, then $A_{\bar{0}}= S'\oplus S$, where $S$ is a semi-simple two-sided ideal of $A_{\bar{0}}$. In addition, we have:\\
$$\begin{tabular}{cc}
(i)\ $S.I=I.S=I$&(ii)\ $dim I=dim S= dim S.A_{\bar{1}}.$
\end{tabular}$$
\end{lemma}

\begin{proof}
By a simple computation, we can see easily that $S'= A_{\bar{0}}\cap J=\left\{x\in A_{0}, B(x,I)=\left\{0\right\}\right\}$ is a two-sided  ideal of $A_{\bar{0}}$. Consequently, there exist a two-sided ideal $S$ of $A_{\bar{0}}$ such that $A_{\bar{0}}=S\oplus S'$. Since $I$ is a minimal two-sided ideal of $A$, then by Proposition \ref{min}, we have $B(S'.I,A)= B(S'.I,A_{\bar{0}})=B(S',I.A_{\bar{0}})= B(S',I)=\left\{0\right\}$ and so $S'.I=\left\{0\right\}$. Similar, we show that $I.S'=\left\{0\right\}$ and then we deduce that $S.I=I.S=I$, which prove the assertion $(i)$. Now we prove $(ii)$. As $B$  is odd and $I\subseteq A_{\bar{1}}$, then $A_{\bar{1}}\subseteq J$. So $A=S\oplus S'\oplus A_{\bar{1}}= S\oplus J$ and consequently $dim S=dim I$. It is clear by using $(i)$ that, $dim I\leq dim S.A_{\bar{1}}$. Finally to finish, we prove that $dim S.A_{\bar{1}}\leq dim S$. To be done, we consider the following linear map $$\nu: S.A_{\bar{1}} \longrightarrow S^*\ \ \ \mbox{defined by}\ \ \ \nu(x)=B(x,.), \ \ \mbox{where}\ \ x\in S.A_{\bar{1}}.$$   As $B(S.A_{\bar{1}}, S'\oplus A_{1})=\left\{0\right\}$, then for $x\in S.A_{\bar{1}}$ we have, $B(x,S)=\left\{0\right\}$ if and only if $B(x,A)=\left\{0\right\}$ and which implies, by the non degeneracy of $B$, that $\nu$ is injective.\\ 
\end{proof}

The following lemma is useful in the proposition below.

\begin{lemma}\label{A1semisimple}
 Let $A=A_{\bar{0}}\oplus A_{\bar{1}}$ be an associative superalgebra such that $A_{\bar{1}}$ is a semi-simple $A_{\bar{0}}$-bimodule, then $A_{\bar{1}}= {A_{\bar{1}}}^{A_{\bar{0}}}\oplus A_{\bar{0}}A_{\bar{1}}\oplus A_{\bar{1}}A_{\bar{0}}$, where ${A_{\bar{1}}}^{A_{\bar{0}}}:= \left\{x\in A_{\bar{1}},\  x.A_{\bar{0}}= \left\{0\right\}=A_{\bar{0}}.x\right\}$. \\
 \end{lemma}
 
  \begin{proof}
 Clearly ${A_{\bar{1}}}^{A_{\bar{0}}}$ is a $A_{\bar{0}}-$sub-bimodule of $A_{\bar{1}}$ and consequently there exist a non-trivial $A_{\bar{0}}-$sub-bimodule $M$ of $A_{\bar{1}}$ such that $A_{\bar{1}}={A_{\bar{1}}}^{A_{\bar{0}}}\oplus M$. First, if $M$ is irreducible, then $M=A_{\bar{0}}M\oplus MA_{\bar{0}}$ because $A_{\bar{0}}M\oplus MA_{\bar{0}}$ is a non-null $A_{\bar{0}}$-sub-bimodule of $M$. In addition, we have $A_{\bar{0}}A_{\bar{1}}=A_{\bar{0}}M$ and $A_{\bar{1}}A_{\bar{0}}=MA_{\bar{0}}$. So the result. If now $M$ is non-irreducible, then $M={{\bigoplus}^n}_{i=1} M_i$, such that $\forall i\in \left\{1,\cdots,n\right\}$, $M_i$ is an irreducible $A_{\bar{0}}$-sub-bimodule of $M$. Applying the first step of this proof to $M_i$ $\forall i\in \left\{1,\cdots,n\right\}$, we obtain that  $A_{\bar{1}}={A_{\bar{1}}}^{A_{\bar{0}}}\oplus A_{\bar{0}}M\oplus MA_{\bar{0}}$ and so $A_{\bar{1}}= {A_{\bar{1}}}^{A_{\bar{0}}}\oplus A_{\bar{0}}A_{\bar{1}}\oplus A_{\bar{1}}A_{\bar{0}}$.\\   
 \end{proof}
 
 \begin{proposition}
Let $(A,B)$ be a $B$-irreducible odd-symmetric non-simple associative superalgebra such that $A_{\bar{0}}$ is a semi-simple associative algebra. Then $A=S\oplus P(S^*)$ is a semi-direct product of a simple associative algebra $S$ by $P(S^*)$ by means of $(L^*,R^*)$.\\
\end{proposition}

\begin{proof}
As $A$ is $B$-irreducible non-simple, then there exist a minimal graded two-sided ideal $I$ of $A$. By Proposition \ref{min}, $I$ is a non-trivial irreducible $A_{\bar{0}}$-sub-bimodule of $A_{\bar{1}}$ such that $I.A_{\bar{1}}=A_{\bar{1}}.I=\left\{0\right\}$. Let $S'=A_{\bar{0}}\cap J$, where $J$ is the orthogonal of $I$ with respect to $B$. By Lemma \ref{SS'}, we have $A_{\bar{0}}= S\oplus S'$, where $S$ is a two-sided semi-simple ideal of $A_{\bar{0}}$ and $I$ is a $S$-bimodule of $A_{\bar{1}}$. Consequently, there exist a $S$-sub-bimodule $M$ of $A_{\bar{1}}$ such that $A_{\bar{1}}=I\oplus M$. We apply  Lemma \ref{A1semisimple}, we obtain that  $A_{\bar{1}}=I\oplus M^s\oplus S.M\oplus M.S$, where $M^s=\left\{x\in M, \, x.S=S.x=\left\{0\right\}\right\}$. Since $S.A_{\bar{1}}=S.I+S.M$, then by Lemma \ref{SS'}, we obtain that  $dim S.M= dim S.A_{\bar{1}} - dim SI= dim S.A_{\bar{1}} - dim I= 0$. Similar we prove that $dim M.S=0$ and we deduce that
 $$A=(S\oplus I)\oplus (S'\oplus M^s).$$
 By a direct computation, we can see easily that $S\oplus I$ is a graded two-sided ideal of $A$. Moreover, since $B(S\oplus I, S'\oplus M^s)=\left\{0\right\}$, then  $S\oplus I$ is non-degenerate. So we deduce that $A= S\oplus I$ because $A$ is $B$-irreducible. As $I$ is a maximal graded ideal of $A$ and according to Lemma \ref{quotient}, we deduce that $S$ is a simple associative algebra. Now, we consider the following map: 
\begin{eqnarray*} 
\Delta&: S\oplus I \longrightarrow S\oplus P(S^*)\\
      & s+i \longmapsto s+B(i,.),
\end{eqnarray*}
where $(S\oplus P(S^*), \star, \widetilde{B})$ is a semi-direct product of $S$ by $P(S^*)$ defined as example 1. The fact that $B$ is odd, non-degenerate and associative, implies that $\Delta$ is an isomorphism of odd-symmetric associative superalgebras. In addition, we have $\widetilde{B}(\Delta(x),\Delta(y))= B(x,y)$.\\     
\end{proof}

The corollary below, give an inductive description of odd-symmetric associative superalgebras whose even parts are semi-simple associative algebras.\\

\begin{corollary}
Let $(A,B)$ be an odd-symmetric associative superalgebra such that $A_{\bar{0}}$ is a semi-simple associative algebra, then $A$ is an orthogonal direct sums of odd-symmetric associative superalgebras $A_i$, $i\in \left\{1,\cdots,n\right\}$ such that $A_i$ is either the superalgebra $Q_n(\K)$ for a suitable $n\in \N^*$ or a semi-direct product of $M_n(\K)$ (where $n\in {\N}^*$) by $P({(M_n(\K))}^*)$ by means of $(L^*,R^*)$.\\ 
\end{corollary}      
     
Now, we consider the general case where $(A,B)$ is an odd-symmetric associative superalgebra whose even part $A_{\bar{0}}$ is a semi-simple $A_{\bar{0}}$-bimodule and we are going to give an inductive description of these superalgebras. We will show that the particular case where $A_{\bar{0}}$ is a semi-simple associative algebra and the example 2 will be very useful in the determination of this inductive description.\\ 

\begin{lemma}\label{ss}
 Let $(A= A_{\bar{0}}\oplus A_{\bar{1}},B)$ be an odd-symmetric associative superalgebra, then: 
 \begin{eqnarray*}
A_{\bar{1}}\ \mbox{is a semi-simple}\ A_{\bar{0}}\mbox{-bimodule if and only if}\ \ \  A_{\bar{0}}\ \mbox{is a semi-simple}\  A_{\bar{0}}\mbox{-bimodule}.\\
\end{eqnarray*}
 \end{lemma}
   
\begin{proof}
 According to Proposition \ref{caracterisationimpaire}, we have $A_{\bar{1}}\cong {A_{\bar{0}}}^*$ as $A_{\bar0}$-bimodules. Using this argument, proving the Lemma \ref{ss} is equivalent to proving that $A_{\bar{0}}$ is a semi-simple $A_{\bar{0}}$-bimodule if and only if ${A_{\bar{0}}}^*$ is a semi-simple $A_{\bar{0}}$-bimodule. We suppose that ${A_{\bar{0}}}^*$ is a semi-simple $A_{\bar{0}}$-bimodule and $M$ is a $A_{\bar{0}}$-sub-bimodule of $A_{\bar{0}}$. In the following we are going to seek for a $A_{\bar{0}}$-sub-bimodule $N$ of $A_{\bar{0}}$ such that $A_{\bar{0}}= M\oplus N$. By a simple computation, we can see that $\tilde{M}:= \left\{f\in {A_{\bar{0}}}^*, {f\mid}_M=0\right\}$ is a $A_{\bar{0}}$-sub-bimodule of ${A_{\bar{0}}}^*$ such that ${A_{\bar{0}}}^*/\tilde{M}\cong M^*$ as vector spaces by means  of the linear map $\alpha: {A_{\bar{0}}}^* \longrightarrow M^*$ defined by $\alpha(f)= {f\mid}_M$. Consequently, we have 
 \begin{eqnarray}
\label{egalite1} dim A_{\bar{0}} &=& dim M + dim \tilde{M}. 
 \end{eqnarray}
 The fact that ${A^*}_{\bar{0}}$ is a semi-simple $A_{\bar{0}}$-bimodule and $\tilde{M}$ is a $A_{\bar{0}}$-sub-bimodule of ${A_{\bar{0}}}^*$, implies the existence of a $A_{\bar{0}}$-sub-bimodule $T$ of ${A_{\bar{0}}}^*$ such that ${A_{\bar{0}}}^*= T\oplus \tilde{M}$. Let $\tilde{T}:= \left\{x\in A_{\bar{0}}, f(x)=0, \, \forall f\in T\right\}$, then we can see easily that $\tilde{T}$ is a $A_{\bar{0}}$-sub-bimodule of $A_{\bar{0}}$ such that $A_{\bar{0}}/\tilde{T}\cong T^*$ as vector spaces by means of the linear map $\beta: A_{\bar{0}}\longrightarrow T^*$ defined by $\beta(x)=\tilde{x}$, where $\tilde{x}(g):= g(x),\, \forall g\in T^*$. Then, we have: 
 \begin{eqnarray}
\label{egalite2} dim A_{\bar{0}} &=& dim T + dim \tilde{T}
 \end{eqnarray}
  
  On the other hand, for $x\in \tilde{T}\cap M$ we have $f(x)=0$ $\forall f\in {A_{\bar{0}}}^*$ because $f= \phi + \varphi$ where $\phi\in T$ and $\varphi\in \tilde{M}$. So  $x=0$ and we deduce that $\tilde{T}\cap M=\left\{0\right\}$. Now, by (\ref{egalite1}), (\ref{egalite2}) and the fact that $\tilde{T}\cap M=\left\{0\right\}$, we infer that $A_{\bar{0}}=\tilde{T}\oplus M$. Conversely, we suppose that $A_{\bar{0}}$ is a semi-simple  $A_{\bar{0}}$-bimodule. Then $A_{\bar{0}}= S_1\oplus \cdots \oplus S_n\oplus Ann(A_{\bar{0}})$, where $S_i$ is a simple two-sided ideal of $A_{\bar{0}}$ and $Ann(A_{\bar{0}})$ is the annihilator of $A_{\bar{0}}$. It is well known that $A_{\bar{0}}$ is a symmetric associative algebra. More precisely, if $f_i(x,y)=tr(R_xR_y)$, $\forall x,y\in S_i$ and $g$ is an arbitrary symmetric non-degenerate bilinear form of $Ann(A_{\bar{0}})$, then the symmetric bilinear form $\gamma$  defined by:
 \begin{eqnarray*}
 \left\{0\right\} &=& {\gamma\vert}_{S_i\times S_j}= {\gamma\vert}_{S_i\times Ann(A_{\bar{0}})},\ \ \forall i,j\in \left\{1,\cdots,n\right\}\ \mbox{such that}\ i\neq j\\
f_i &:=& {\gamma\mid}_{S_i\times S_i}\ \ \ \ \mbox{and}\ \ \ \ g :={\gamma\mid}_{Ann(A_{\bar{0}})\times Ann(A_{\bar{0}})}, 
 \end{eqnarray*}
 is non-degenerate and associative bilinear form on $A_{\bar{0}}$. So $(A_{\bar{0}}, \gamma)$ is a symmetric associative algebra. Consequently, by Proposition \ref{caracterisationpaire}, we deduce that $A_{\bar{0}}\cong {A^*}_{\bar{0}}$ by means of $\phi_{\bar{0}}$ defined by $\phi_{\bar{0}}(x):= \gamma(x,.),\ \forall x\in A_{\bar{0}}$ and we conclude that ${A^*}_{\bar{0}}$ is a semi-simple $A_{\bar{0}}$-bimodule.\\
 \end{proof} 

\begin{proposition}\label{pdtdansA}
Let $(A,B)$ be an odd-symmetric associative superalgebra such that $A_{\bar{0}}$ is a semi-simple $A_{\bar{0}}-$bimodule, then:
\begin{eqnarray} 
A&:=& S_1\oplus \cdots \oplus S_n\oplus N\oplus {\hat{S}}_1\oplus \cdots \oplus {\hat{S}}_n\oplus \hat{N},
\end{eqnarray}
where $S_i$ (resp. $\hat{S}_i$) is a simple two-sided ideal (resp. an irreducible $A_{\bar{0}}$-bimodule) of $A_{\bar{0}}$ (resp. $A_{\bar{1}}$) $\forall i\in \left\{1, \cdots, n\right\}$, $N=Ann(A_{\bar{0}})$ and $\hat{N}$ is a trivial  $A_{\bar{0}}$-bimodule of $A_{\bar{1}}$ (i.e $A_{\bar{0}}.\hat{N}= \left\{0\right\}=\hat{N}.A_{\bar{0}}$). Moreover, we have:
\begin{enumerate}
\item[(i)] ${\hat{S}}_i{\hat{S}}_j= \left\{0\right\}= {\hat{S}}_j{\hat{S}}_i$, for $i,j\in \left\{1,\cdots,n\right\}$ such that $i\neq j$,
\item[(ii)] ${\hat{S}}_i{\hat{S}}_i\in \left\{\left\{0\right\}, S_{i} \right\}$, $\forall i\in \left\{1,\cdots, n\right\}$,
\item[(iii)] $\hat{N} {\hat{S}}_i=\left\{0\right\}=  {\hat{S}}_i \hat{N}$, $\forall i\in \left\{1, \cdots, n\right\}$,
\item[(iv)] $\hat{N} \hat{N}\subseteq N$.
\end{enumerate}
\end{proposition} 

\begin{proof}
According to Lemma \ref{ss}, we have $(A_{\bar{0}}= S_1\oplus \cdots \oplus S_n\oplus N,\gamma)$ is a symmetric associative algebra, where $S_i$, for all $i\in \left\{1,\cdots,n\right\}$, is a simple two-sided ideal of $A_{\bar{0}}$ and $N$ its annihilator. Consider the two isomorphisms of $A_{\bar{0}}$-bimodules  $\phi_{\bar{0}}: A_{\bar{0}}\longrightarrow A_{\bar{0}}^*$ and $\phi: A_{\bar{1}}\longrightarrow A_{\bar{0}}^*$ defined respectively by $\phi_{\bar{0}}(x):= \gamma(x,.)$ and $\phi(y):= B(y,.)$, $\forall x\in A_{\bar{0}},y\in A_{\bar{1}}$. Consequently, the map $\eta$ defined by  $\eta:= {\phi}^-\circ \phi_{\bar{0}}$ is an isomorphism of $A_{\bar{0}}$-bimodule such that: $\forall x\in A_{\bar{0}}, \, \eta(x)=\hat{x}$, where $\hat{x}\in A_{\bar{1}}$ and $\gamma(x,x')=B(\hat{x},x'),\, \forall x'\in A_{\bar{0}}$. So, we obtain that $A= S_1\oplus \cdots \oplus S_n\oplus N\oplus {\hat{S}}_1\oplus \cdots \oplus {\hat{S}}_n\oplus \hat{N}$, where ${\hat{S}}_i:= \eta(S_i)$ is an irreducible $A_{\bar{0}}$-sub-bimodule of $A_{\bar{1}}$ and $\hat{N}:= \eta(N)$ is a  trivial $A_{\bar{0}}$-sub-bimodule of $A_{\bar{1}}$. In addition, $$B(S_i, \hat{S}_j)= \left\{0\right\}, \ \ B(S_i,\hat{N})= B(N, \hat{S}_i)=\left\{0\right\}, \ \forall\, i,j\in \left\{1,\cdots,n\right\}, \ \mbox{such that}\ i\neq j.$$ 
By a simple computation, where we use the associativity and the non-degeneracy of $B$ and the fact that $\hat{S}_i$ is an irreducible $A_{\bar{0}}$-bimodule of $A_{\bar{1}}$, we verify easily that:
\begin{eqnarray*}
  N &\subseteq& Ann(A),\\
  \left\{0\right\}&=& S_i.\hat{N}=\hat{N}.S_i,\, \forall i\in \left\{1,\cdots,n\right\},\\
 \left\{0\right\}&=& S_i.\hat{S}_j=\hat{S}_j.S_i,\ \forall i,j\in \left\{1,\cdots,n\right\}\ \ \mbox{such that}\ i\neq j,\\
 \hat{S_i}&=&  S_i.\hat{S}_i=\hat{S}_i.S_i,\ \forall i\in \left\{1,\cdots, n\right\}.
\end{eqnarray*}   
Now using these identity, we prove $(i)$-$(iv)$ of Proposition \ref{pdtdansA}. Indeed, let $i,j\in \left\{1,\cdots,n\right\}$ such that $i\neq j$, then $\hat{S}_i\hat{S}_j=(\hat{S}_iS_i)\hat{S}_j= \hat{S}_i(S_i\hat{S}_j)=\left\{0\right\}$ and in the same way, we have $\hat{S}_j\hat{S}_i=\left\{0\right\}$. On the other hand, $B(\hat{S}_i\hat{S}_i, \hat{S}_j)=B(\hat{S}_i\hat{S}_i, S_j\hat{S}_j)=B(\hat{S}_i(\hat{S}_iS_j), \hat{S}_j)=\left\{0\right\}$ and $B(\hat{S}_i\hat{S}_i, \hat{N})=B(\hat{S}_i(\hat{S}_iS_i), \hat{N})=B(\hat{S}_i\hat{S}_i,S_i\hat{N})=\left\{0\right\}$, so $\hat{S}_i\hat{S}_i\subseteq S_i$. As $\hat{S}_i \hat{S}_i$ is an ideal of $S_i$ and $S_i$ is simple, we obtain that $\hat{S}_i \hat{S}_i= S_i$ or $\hat{S}_i \hat{S}_i= \left\{0\right\}$. Now, $\hat{N}\hat{S}_i=\hat{N}(S_i\hat{S}_i)=(\hat{N}S_i)\hat{S}_i=\left\{0\right\}$ and in the same way, we show that $\hat{S}_i\hat{N}=\left\{0\right\}$. Finally, $B(\hat{N}\hat{N}, \hat{S_i}) =B(\hat{N}(\hat{N}S_i), \hat{S_i})=\left\{0\right\}$ and which implies that $\hat{N}\hat{N}\subseteq N$.\\
\end{proof}

Following Proposition \ref{pdtdansA}, we conclude that any odd-symmetric associative superalgebra $A$ whose even part $A_{\bar{0}}$ is a semi-simple $A_{\bar{0}}$-bimodule is an orthogonal direct sums of graded non-degenerate two-sided ideals. More precisely:
$$A= (S_1\oplus \hat{S}_1)\oplus \cdots \oplus (S_n\oplus \hat{S}_n)\oplus (N\oplus \hat{N}),$$
such that $\forall i\in \left\{1,\cdots,n\right\}$, $I_i:=S_i\oplus \hat{S}_i$ is $B_i$-irreducible, where $B_i:= {B\mid}_{I_i\times I_i}$. In particular:
\begin{enumerate}
 \item[] If $\hat{S}_i\hat{S}_i=S_i$, we can see easily that $I_i$ is simple and so it is isomorphic to $Q_n(\K)$ with a suitable $n\in \N^*$.
  \item[] Otherwise (i.e  $\hat{S}_i\hat{S}_i=\left\{0\right\}$), $I_i:= S_i\oplus P({{S}^*}_i)$ is the semi-direct product of $S_i$ by $P({{S}^*}_i)$ by means of $(L^*,R^*)$ (for definition of $(L^*,R^*)$, see example (1)).
\end{enumerate}  
   On the other hand, the ideal $ N\oplus \hat{N}:= L_1\oplus \cdots \oplus L_m$, where for all $ j\in \left\{1,\cdots ,m\right\}$, $L_j= {(L_j)}_{\bar{0}}\oplus {(L_j)}_{\bar{1}}$ are $B_j$-irreducible graded two-sided ideals, with $B_j:= {B\mid}_{L_j\times \hat{L_j}}$. In particular: 
   \begin{enumerate}
\item[] If ${(L_j)}_{\bar{1}}.{(L_j)}_{\bar{1}}= \left\{0\right\}$, then $L_j$ is the 2-dimensional odd-symmetric associative superalgebra with null product.
\item[] Otherwise (i.e ${(L_j)}_{\bar{1}}.{(L_j)}_{\bar{1}}\neq \left\{0\right\}$), then $L_j:= P({{(L_j)}_{\bar{1}}}^*)\oplus {(L_j)}_{\bar{1}}$ is the generalized semi-direct product of $P({{(L_j)}_{\bar{1}}}^*)$ by ${(L_j)}_{\bar{1}}$ by means of ${\gamma}_j:  {(L_j)}_{\bar{1}}\times  {(L_j)}_{\bar{1}}\longrightarrow P({{(L_j)}_{\bar{1}}}^*)$ defined by ${\gamma}_j(x,y)= B_j(x.y,.),\ \forall x,y\in {(L_j)}_{\bar{1}}$. 
\end{enumerate}
In conclusion, any odd-symmetric associative superalgebra $(A,B)$ whose even part $A_{\bar{0}}$ is a semi-simple $A_{\bar{0}}$-bimodule is an orthogonal direct sums of odd-symmetric associative superalgebras $A_i$, $i\in \left\{1,\cdots,n\right\}$ such that $A_i\in \left\{ \left\{0\right\}, {\cal R},  Q_n(\K), M_n(\K)\oplus P({(M_n(\K))}^*), P(M^*)\oplus M\right\}$ where ${\cal R}$ is the 2-dimensional odd-symmetric associative superalgebra with null product, $M_n(\K)\oplus P({(M_n(\K))}^*)$, for a suitable $n\in \N^*$, is the semi-direct product of  $M_n(\K)$ by $P({(M_n(\K))}^*)$ by means of $(L^*,R^*)$ and $P(M^*)\oplus M$ is the generalized semi-direct product of an associative superalgebra $M=M_{\bar{1}}$ by $P(M^*)$ by means of $\gamma$ defined as example 2.\\

\subsection{Even-symmetric associative superalgebras whose even parts are semi-simple bimodules}
We start by characterizing the minimal graded two-sided ideal of an even-symmetric associative superalgebra whose even part is a semi-simple bimodule.\\
 
\begin{proposition}\label{idealminimal}
Let $(A,B)$ be an even-symmetric $B$-irreducible non simple associative superalgebra. If we suppose that $A$ is different from  the one-dimensional algebra with null product such that $A_{\bar{0}}$ is a semi-simple $A_{\bar{0}}$-bimodule then, $I$ is a minimal graded two-sided ideal of $A$ if and only if $I=\K x$ where $x$ is a homogeneous element of $Ann(A)$.\\ 
\end{proposition}

\begin{proof}
Let $I$ be a minimal graded two-sided ideal of $A$ and $J$ its orthogonal  with respect to $B$. The fact that $A$ is  $B$-irreducible, non-simple and different from the one-dimensional algebra with null product implies that $I$ is totally isotropic. According to Lemma \ref{Iideal}, $I$ is with null product. As $I$ is minimal, then $A/J$ is a one-dimensional super-algebra or a simple associative superalgebra. In the first step, we suppose that $A/J$ is a simple associative superalgebra. 
 As the minimal graded two-sided ideal $I$ is with null product, then $I_{\bar{0}}=I\cap Ann(A_{\bar{0}})$ since the $A_{\bar{0}}$-bimodule $A_{\bar{0}}$ is semi-simple i.e $A_{\bar{0}}= S\oplus Ann(A_{\bar{0}})$ where $S$ is the greatest semi-simple two-sided ideal of $A_{\bar{0}}$. Using now the associativity (resp. the parity ) of $B$, we obtain that $B(S, Ann(A_{\bar{0}}))= \left\{0\right\}$ (resp. $B(S, A_{\bar{1}})=\left\{0\right\}$). So, it follows that $S\subseteq J$ and we deduce that ${(A/J)}_{\bar{0}}\cong Ann(A_{\bar{0}})/J\cap Ann(A_{\bar{0}})$ as associative algebras. This contradicts the fact that ${({(A/J)}_{\bar{0}})}^2\neq \left\{0\right\}$ since  $A/J$ is a simple associative superalgebra. In the second step, we consider $A/J$ is a one-dimensional superalgebra with null product. Then, by Lemme \ref{quotient} $(ii)$, we have $I= \K x$ where  $x$ is a homogeneous element of $Ann(A)$.\\  
\end{proof}

\begin{proposition}\label{annulateurnul}
Let $(A,B)$ be an even-symmetric associative superalgebra such that $A_{\bar{0}}$ is a semi-simple $A_{\bar{0}}$-bimodule. If $Ann(A)= \left\{0\right\}$, then $A$ is orthogonal direct sums of simple even-symmetric associative superalgebras.\\   
\end{proposition}  
  
\begin{proof}
 In order to prove this proposition, we consider the two following cases: \\
  
\textbf{First case:} We suppose that $A$ is $B$-irreducible such that $Ann(A)= \left\{0\right\}$. If $A_{\bar{1}}=\left\{0\right\}$, then $A= A_{\bar{0}}$ and consequently $A$ is a simple associative algebra. In the following, we suppose that $A_{\bar{1}}\neq \left\{0\right\}$. If in addition we suppose that $A$ is not simple, then  there exist a non-null minimal graded two-sided ideal $I$ of $A$. Following Proposition \ref{idealminimal}, $I= \K x$ where $x\in Ann(A)$. But the fact that $Ann(A)=\left\{0\right\}$, implies that $I=\left\{0\right\}$ and which contradict the fact that $I$ is minimal. So we have $A$ is simple.\\
 
\textbf{Second case:} We suppose now that $A$ is not $B$-irreducible such that $Ann(A)=\left\{0\right\}$. By Proposition \ref{depimp}, we have  $A={{\oplus}^{n}}_{k=1} A_k$, where $\forall k\in \left\{1,\cdots,n\right\}$, $A_k$ is an even-symmetric  $B$-irreducible associative superalgebra such that ${(A_k)}_{\bar{0}}$ is a semi-simple ${(A_k)}_{\bar{0}}$-bimodule and $Ann(A_k)=\left\{0\right\}$. Using the first case of this proof, we obtain that $\forall k\in \left\{1,\cdots,n\right\}$, $A_k$ is a simple even-symmetric associative superalgebra.\\   
\end{proof}

In the following, our next goal is to give an inductive description of even-symmetric associative superalgebras whose even parts are semi-simple $A_{\bar{0}}$-bimodules. By means of the proposition above, we reduce the study of these superalgebras to those that are $B$-irreducible with non-null annihilator. In \cite{Im}, authors introduced the notion of generalized double extension of even-symmetric associative superalgebra by a one-dimensional superalgebra with null product in order to give an inductive description of even-symmetric Novikov superalgebras. We recall This notion, since it will be very useful in the sequel. Let $(A,*,B)$ be an even-symmetric associative superalgebra , $\K e$ a one-dimensional superalgebra with null product, $k\in {\K}_{\mid e\mid}$ (we recall that: ${\K}_{\bar{0}}=\K$ and ${\K}_{\bar{1}}=\left\{0\right\}$), $x_0\in A_{\bar{0}}$ and $D\in {(End(A))}_{\mid e\mid}$ such that:  
  \begin{eqnarray}
 D(x_0)=G(x_0),\ \ D^2(x)=x_0*x, \ \ D\circ G= G\circ D,\ \  D(x*y)=D(x)*y, \ \  \forall x,y\in A,\label{conditionpourdim1}
 \end{eqnarray}
 where  $G\in {(End(A))}_{\mid e\mid}$ which satisfies $B(G(x), y)= B(x, D(y)), \ \forall\, x,y\in A$. The $\Z_{2}$-graded vector space $\tilde{A}:= \K e^*\oplus A\oplus \K e$ endowed with the following product: 
  \begin{eqnarray}
\nonumber e^*\in Ann(A),\ \ e.e = x_0 + k e^*,\ \  x.y = x*y + {(-1)}^{\mid e\mid} B(D(x),y)e^*,\\ 
  e.x= D(x) + {(-1)}^{\mid e\mid} B(x,x_0)e^*, \ \  x.e = G(x) + B(x,x_0)e^*, \ \ \forall x,y\in A.\label{produitavecunedimension}
  \end{eqnarray}
and the following bilinear form:   
\begin{eqnarray}
{\tilde{B}\mid }_{A\times A}:= B, \ \ \tilde{B}(e^*, e)=1, \ \ \tilde{B}(e^*, \K e^*\oplus A)= \left\{0\right\}= \tilde{B}(e, A\oplus \K e).
 \end{eqnarray}
 
 is an even-symmetric associative superalgebra. According to \cite{Im}, $\tilde{A}$ (or $(\tilde{A},\tilde{B})$) is called the even generalized  double extension of $A$ by $\K e$ by means of $(D,x_0,k)$ (resp. the odd generalized double extension of $A$ by $\K e$ by means of $(D,x_0)$) if $e$ is an even (resp. an odd) homogeneous element.  In particular, we have $G= D^*$, where $D^*$ is the adjoint of $D$ with respect to $B$, in the case of the even generalized double extension.\\

 \begin{remark}\label{rmqdeg}
 We remark that if $x_0=0$ and $k=0$, the vector space $\K e$ will be a sub-superalgebra of $\tilde{A}$. Consequently $\tilde{A}$ will be obtained by a central extension of $A$ by $\K e^*$ and a semi-direct product of $\K e^*\oplus A$ by $\K e$ (for definition of central extension and semi-direct product we can see \cite{Im}). In this case the even (resp. odd) generalized double extension will be called the even (resp. odd) double extension.\\
 \end{remark}

 Now we suppose that $A$ is an even-symmetric associative superalgebra such that $A_{\bar{0}}$ is a semi-simple $A_{\bar{0}}-$bimodule. Following the product (\ref{produitavecunedimension}) above, we notice that the even part ${\tilde{A}}_{\bar{0}}$ of the odd generalized double extension  coincide with $A_{\bar{0}}$ and so it is a semi-simple ${\tilde{A}}_{\bar{0}}$-bimodule. Whereas, the even part ${\tilde{A}}_{\bar{0}}$ of the even generalized double extension is not in general a semi-simple ${\tilde{A}}_{\bar{0}}$-bimodule. In order to give an inductive description of even-symmetric associative superalgebras whose even parts are semi-simple $A_{\bar{0}}$-bimodules, we introduce a particular notion of the even double extension which is named the elementary even double extension. Let $(A,*,B)$ be an even-symmetric associative superalgebra such that $A_{\bar{0}}$ is a semi-simple $A_{\bar{0}}-$bimodule, $\K e$ the one-dimensional algebra with null product and $D$ an even homomorphism of $A$ such that:

\begin{eqnarray}
{D\mid}_{A_{\bar{0}}}&=&0,\ D^2=0, \ \ D\circ D^*= D^*\circ D,  \nonumber\\
D(A_{\bar{1}})*A_{\bar{1}}&=&\left\{0\right\},\ D(A_{\bar{0}}*A_{\bar{1}})=\left\{0\right\},\ D(A_{\bar{1}}*A_{\bar{0}})= D(A_{\bar{1}})*A_{\bar{0}},\label{elementaire}.
  \end{eqnarray}
  where $D^*$ is the adjoint of $D$ with respect of $B$. The vector space $\overline{A}:= \K e^*\oplus A\oplus \K e$ endowed with the following product: 
  \begin{eqnarray}
 e^*&\in& Ann(\overline{A}),\ \ e\star e=0,\nonumber\\
  e\star x&=& D(x),\ \ x\star e= D^*(x),\ \ x\star y= x*y + B(D(x),y)e^*,\label{produitelementaire}
  \end{eqnarray}
  and the following supersymmetric bilinear form:
  \begin{eqnarray} 
 {\overline{B}\mid }_{A\times A}&:=& B, \ \ \overline{B}(e^*, e)=1, \ \ \overline{B}(e^*, \K e^*\oplus A)= \left\{0\right\}= \overline{B}(e, A\oplus \K e),   
  \end{eqnarray}
  is an even-symmetric superalgebra such that ${\overline{A}}_{\bar{0}}$ is a semi-simple  ${\overline{A}}_{\bar{0}}$-bimodule. $\overline{A}$ is called the elementary even double extension of $A$ by $\K e$ by means of $D$.\\
  
  We remark that any even homomorphism $D$ of an even-symmetric associative superalgebra $(A,B)$ which satisfies the conditions (\ref{elementaire}) it satisfies also the conditions (\ref{conditionpourdim1}) with $G= D^*$, where $D^*$ is the adjoint of $D$ with respect to $B$, and $x_0=0$. So, we can deduce that the elementary even double extension of $(A,B)$ by the one-dimensional algebra $\K e$ with null product by means of $D$ is the even double extension of $(A,B)$ by $\K e$ by means of $D$.\\

 \begin{lemma}
  Let $(A,B)$ be an even-symmetric $B$-irreducible associative superalgebra which is different of the one-dimensional algebra with null product. We suppose that $A_{\bar{0}}$ is a semi-simple $A_{\bar{0}}$-bimodule such that $Ann(A)\cap A_{\bar{0}}\neq \left\{0\right\}$, then $Ann(A)\cap A_{\bar{0}}$ is enclosed strictly in $Ann(A_{\bar{0}})$ (i.e $Ann(A)\cap A_{\bar{0}} \subsetneq Ann(A_{\bar{0}})$).\\
  \end{lemma}    
    \begin{proof}
  It is clear that $Ann(A)\cap A_{\bar{0}} \subseteq Ann(A_{\bar{0}})$. If we suppose that $Ann(A)\cap A_{\bar{0}}=Ann(A_{\bar{0}})$, then $Ann(A_{\bar{0}})$ is non-null graded two-sided ideal of $A$. The fact that $A$ is $B$-irreducible and different of the one-dimensional algebra with null product, imply that $Ann(A_{\bar{0}})$ is a degenerate ideal of $A$, i.e there exist $x\in Ann(A_{\bar{0}})\setminus \left\{0\right\}$ such that $B(x,Ann(A_{\bar{0}}))= \left\{0\right\}$. Consequently, $B(x,A)= B(x,A_{\bar{0}})= B(x, S\oplus Ann(A_{\bar{0}}))= \left\{0\right\}$. We deduce that $x$ is in the orthogonal of $A$ with respect to $B$ and which contradict the fact that $B$ is non-degenerate.\\ 
  \end{proof}
  
   \begin{proposition}\label{doubleextension}
  Let $(A,.,B)$ be an even-symmetric $B$-irreducible associative superalgebra which is different of the one-dimensional algebra with null product and such that $A_{\bar{0}}$ is semi-simple $A_{\bar{0}}$-bimodule. If $Ann(A)\cap A_{\bar{0}}\neq \left\{0\right\}$ (resp. $Ann(A)\cap A_{\bar{1}}\neq \left\{0\right\}$), then $A$ is an elementary even double extension (resp. an odd generalized double extension) of an even-symmetric associative superalgebra $W$ with even part $W_{\bar{0}}$ is a semi-simple $W_{\bar{0}}$-bimodule by the one-dimensional algebra with null product (resp. by the one-dimensional superalgebra with null even part).\\   
  \end{proposition}
  
   \begin{proof}
\textbf{First step:} We suppose that $Ann(A)\cap A_{\bar{0}}\neq \left\{0\right\}$, $I:= \K e^*$, where $e^*\in Ann(A)\cap A_{\bar{0}}\setminus \left\{0\right\}$ and $J$ its orthogonal with respect to $B$. It is clear by the associativity (resp. the parity) of $B$ that $S\subseteq J$ (resp. $A_{\bar{1}}\subseteq J$) where $S$ is the greatest semi-simple ideal of $A_{\bar{0}}$. Consequently by the minimality of $I$ and the non-degeneracy of $B$, there exist $e\in Ann(A_{\bar{0}})$ such that $A= J\oplus \K e$ and $B(e^*,e)\neq 0$. As $e\in Ann(A_{\bar{0}})$, then $V:= \K e$ is the sub-superalgebra of $A$ and we deduce by Theorem V.2 of \cite{Im} and Remark  \ref{rmqdeg}, that $A$ is an even double extension of the even-symmetric associative superalgebra $W:=J/I$ by $V = \K e$ with null product by means of the homomorphism $D$ of $W$ defined by $D(\bar{x}):= \overline{e.x}$, $\forall x\in J$. Using the associativity of $A$  and the fact that $e\in Ann(A_{\bar{0}})$, we can see easily that $D$ satisfies conditions (\ref{elementaire}). So it comes that $A$ an elementary even double extension of $W$  by $\K e$ by means of $D$. In addition, it is clear that the even part $W_{\bar{0}}$ of the even-symmetric associative superalgebra $W$ is a semi-simple $W_{\bar{0}}$-bimodule.\\
  
\textbf{Second step:} Supposing now that $Ann(A)\cap A_{\bar{1}}\neq \left\{0\right\}$,  $I:= \K e^*$, where $e^*\in Ann(A)\cap A_{\bar{1}}\setminus \left\{0\right\}$ and $J$ its orthogonal with respect to $B$. Similar to the first step of this proof, there exist $e\in A_{\bar{1}}$ such that  $A= J\oplus \K e$ and $B(e^*, e)\neq 0$. Applying Theorem V.2 of \cite{Im}, we obtain that  $A$ is an odd generalized double extension of the even-symmetric associative superalgebra $W:=J/I$ by the superalgebra $V := \K e$ with null product by means of the homomorphism $D$ of $W$ defined by  $D(\bar{x}):= \overline{e.x}$, $\forall x\in J$ and $x_0:= e.e\in W_{\bar{0}}$. Moreover, it is trivially that the even-part  $W_{\bar{0}}$ of the even-symmetric associative superalgebra $W$ is a semi-simple $W_{\bar{0}}$-bimodule since it coincide with $A_{\bar{0}}$.\\
  \end{proof}
   
  \begin{remark}\label{dimensiondeux}
 Any two-dimensional even-symmetric associative superalgebra $A= A_{\bar{0}}\oplus A_{\bar{1}}$ is an even-symmetric associative algebra or the even-symmetric associative superalgebra (i.e $A_{\bar{0}}=\left\{0\right\}$). If we suppose in addition that $A_{\bar{0}}$ is a semi-simple $A_{\bar{0}}$-bimodule, then:
\begin{enumerate}
\item[] If  $A$ is $B$-irreducible, we have $A$ the elementary even double extension of $\left\{0\right\}$ by the one-dimensional algebra with null product or the odd generalized double extension of $\left\{0\right\}$ by the superalgebra of one-dimensional with null even part.
\item[] Otherwise $A$ is orthogonal direct sum of two copy of  $\K$ or an orthogonal direct sum of the $\K$ with the one-dimensional algebra with null product.
\end{enumerate}   
 \end{remark}

Let $\cal{O}$ be the set consisting of $\left\{0\right\}$, the one-dimensional algebra with null product and $M_{r,s}(\K)$ where $r\geq 1$ and $s\geq 0$.\\
 
 \begin{corollary}\label{desinduc}
  Let $(A,.,B)$ be an even-symmetric associative superalgebra such that $A_{\bar{0}}$ is a semi-simple $A_{\bar{0}}$-bimodule. If $A\notin {\cal O}$, then $A$ is obtained from a finite number of element of $\cal{O}$ by a finite sequence of elementary even double extensions and/or odd generalized double extensions and/or orthogonal direct sums.\\
 \end{corollary} 

\begin{proof}
  We proceed by induction on the dimension of $A$. If $dim(A)=1$, then $A$ is either the field $\K$ or the one-dimensional Lie algebra with null product and in both cases $A\in {\cal O}$. If $dim(A)=2$, then by Remark \ref{dimensiondeux}, we have $A$ is obtained as Corollary \ref{desinduc}. Suppose that Corollary \ref{desinduc} is true for $dim (A)<  n$ with $n\geq 2$. We consider $dim (A)=n$. We have to analyze two cases: \\
  
\textbf{First case:} Suppose that $A$ is $B$-irreducible. If $Ann(A)=\left\{0\right\}$, then by Proposition \ref{annulateurnul}, $A$ is simple. Whereas, if $Ann(A)\neq \left\{0\right\}$, by Proposition \ref{doubleextension}, $A$ is an elementary even double extension (resp. odd generalized double extension ) of an even-symmetric associative superalgebra $W$ such that its even part $W_{\bar{0}}$ is a semi-simple $W_{\bar{0}}$-bimodule ( $dim(W)= dim(A)-2$) by the one-dimensional algebra with null product (resp. by the one-dimensional superalgebra ). Applying the induction hypothesis to $W$, we infer the corollary for $A$.\\
  
\textbf{Second case:} Now we suppose that $A$ is not $B$-irreducible. In view of Proposition \ref{depimp}, $A= {{\bigoplus}^m}_{k=1}A_k$ where $\left\{A_k,\  1\leq k\leq m\right\}$ is a set of $B$-irreducible graded ideals of $A$ such that $B(A_k, A_k')= \left\{0\right\}$, $\forall k\neq k'$. It is clear that ${(A_k)}_{\bar{0}}$ is a semi-simple ${(A_k)}_{\bar{0}}$-bimodule and $dim(A_k)< n$, $\forall k\in \left\{1,\cdots , m\right\}$. Applying the induction hypothesis to $A_k$ with $k\in \left\{1,\cdots, m\right\}$.\\
\end{proof}

\section{Inductive description of associative superalgebras with homogeneous symmetric structure}

This section will detail how an associative superalgebra with homogeneous symmetric structure can be obtained.

\subsection{Inductive description of even-symmetric associative superalgebras}

We give an inductive description of even-symmetric associative superalgebras by using the notion of generalized double extension of even-symmetric associative superalgebras. This notion was introduced in \cite{Im} in order to give an inductive description of symmetric Novikov superalgebras. We begin by recalling some definitions and results established in \cite{Im}. Let $(W, \tilde{B})$ be an even-symmetric associative superalgebra, $V$ an associative superalgebra, $\mu  : V \rightarrow End(W)$ a linear map, $\lambda : V\times V \rightarrow W$ a bilinear map and $\gamma:  V\times V \rightarrow V^{*}$ a bilinear map. According to \cite{Im}, the set composed by $\left\{W , \tilde{B} ,  V , \mu , \lambda  , \gamma\right\}$ is called a context of generalized double extension of $W$ by $V$ if the following conditions are satisfied: 
\begin{eqnarray*}
\mu' (v)( x * y)&=&x * \mu' (v)(y) \\ 
\mu'(v')\circ \mu (v)&=&\mu (v)\circ \mu' (v')\\
\mu'(v \star v')(x)&=&(\mu' (v')\circ \mu' (v))(x) - x * \lambda(v, v')\\ 
\mu' (v'')(\lambda(v, v'))&=&\mu (v)(\lambda(v', v'')) + \lambda(v, v'\star v'') -  \lambda(v \star v', v'')
\end{eqnarray*}
 \begin{eqnarray*}
 \gamma(v, v') (v'') = (-1)^{\mid v\mid (\mid v'\mid + \mid v''\mid )} \gamma(v', v'')(v)
 \end{eqnarray*}
and
\begin{eqnarray*}
 &-& \widetilde{B}(\lambda(v , v') , \lambda( v'' , v'''))+(-1)^{\mid v\mid (\mid v'\mid + \mid v''\mid + \mid v'''\mid )}\widetilde{B}(\lambda(v', v''), \lambda(v''', v))\nn\\ 
&=&\gamma(v \star v' , v'')(v''')+ \gamma(v , v')(v''\star v''') - \gamma(v , v'\star v'')(v''')\nn\\ 
&-&(-1)^{\mid v\mid (\mid v'\mid + \mid v''\mid + \mid v'''\mid)} \gamma (v', v'')(v'''\star v)
\end{eqnarray*}
where $\mu': V \longrightarrow End(W)$ such that $\widetilde{B}(\mu(v)(x) , y )  = (-1)^{\mid v\mid (\mid x\mid + \mid y\mid)} \widetilde{B}( x , \mu' (v)(y)), \,\,\,\forall v\in {\cal V}_{\mid v\mid}, \,\, x\in {\cal W}_{\mid x\mid}, y\in {\cal W}$.\\ 

\begin{theorem}\cite{Im}
Let $\left\{W , \tilde{B} , V , \mu  , \lambda  , \gamma \right\}$ be a context of generalized double extension of the even-symmetric associative superalgebra $(W,\tilde{B})$ by the associative superalgebra $V$. The $\Z_{2}$-graded vector space $A:= V\oplus W\oplus  V^*$ endowed with the following product:
\begin{eqnarray*}
(v + x + f) . (w +  y + g) &=& v \star w + \lambda(v, w) + \gamma(v, w) +  x * y + \phi(x, y)
+ \mu (v)(y)\\ &+& \nu(v, y) + (-1)^{\mid x\mid  \mid y\mid} g \circ \widetilde{R}_{v}  +  \mu' (w)(x) + \nu' (x, w) +  f \circ L_{w} 
\end{eqnarray*}
 and the following supersymmetric bilinear form $B$ such that:   
\begin{eqnarray*}
B( v + x + f ,  w + y + g) &:=&  \widetilde{B}(x, y) + (-1)^{\mid x\mid  \mid y\mid} g(v) + f(w),
\end{eqnarray*}
  $\forall\ (v  + x + f) \in A_{\mid x\mid}$, $(w +  y + g)\in  A_{\mid y\mid}$ is  an even-symmetric associative superalgebra such that $V^{*}$ is a totally isotropic graded two-sided ideal of $A$ and $W\oplus V^{*}$ its orthogonal with respect to $B$. The even-symmetric associative superalgebra $(A,B)$ is called the generalized double extension of the even-symmetric associative superalgebra $(W,\tilde{B})$ by the associative superalgebra $V$ by means of $(\mu, \lambda, \gamma)$. \\ 
\end{theorem}

\begin{remark}
 If $V:=\K e$ is a one-dimensional superalgebra, then the context $\left\{W , \tilde{B} , V , \mu  , \lambda  , \gamma \right\}$ of the generalized double extension of $W$ by $V$ is defined by $\left\{W , \tilde{B} , \K e,  D,  x_0 , \gamma_0 \right\}$ where $D  = \mu (e)\in {(End(W))}_{\mid e\mid }$, $x_0 = \lambda (e, e)\in { W}_{\overline{0}}$ and $\gamma_0  = \gamma(e, e)(e)\in \K_{\mid e\mid}$. In this case, the generalized double extension of $W$ by $V$ by means of $(\mu, \lambda, \gamma)$ is called the even (resp. odd) generalized double extension of $W$ by the $\K e$ by means of $(D,x_0,k)$ (resp. $(D,x_0)$) if $e$ is an even (resp. odd) homogeneous element.\\
\end{remark} 

The following theorem is the converse of the last theorem. 

\begin{theorem}\cite{Im}\label{descriptionnov}
Let $(A,.,B)$ be an even-symmetric $B$-irreducible non-simple associative superalgebra with dimension strictly upper than 1, $I$ a totally isotropic graded two-sided ideal of $A$ and $J$ its orthogonal with respect to $B$. We define:
\begin{enumerate} 
\item[(1)] The even-symmetric associative superalgebra $(A_{1}, Q) := (J/ I,Q)$ such that $Q$ is defined by $Q(\bar{x}, \bar{y}):= B(x,y),\ \forall x,y\in J$.
\item[(2)] the associative superalgebra $A_{2} := A/ J$.
\end{enumerate}
Then $A$ is isomorphic to the generalized double extension of $(A_{1}, Q)$ by $A_{2}$ by means of $(\mu, \lambda , \gamma)$, where $\mu : A_{2} \rightarrow {End (A_{1})}$ is an even linear map, $\lambda : A_{2}\times A_{2} \rightarrow A_{1}$ is an even bilinear map and $\gamma :A_{2}\times A_{2} \rightarrow { A_{2}}^{*}$ is an even bilinear map.\\ 
\end{theorem}

\begin{proposition}
Let $(A,B)$ be an even-symmetric $B$-irreducible non-simple associative superalgebra such that $dim A>1$, then $A$ is a generalized double extension of an even-symmetric associative superalgebra by an element of $\left\{M_{r,s}(\K), Q_n(\K)\right\}$ or an even generalized double extension of an even-symmetric associative superalgebra by the one-dimensional algebra with null product or an odd generalized double extension of an even-symmetric associative superalgebra by the one-dimensional superalgebra with null even part.\\ 
\end{proposition}

\begin{proof}
We consider the following two cases:\\ 

\textbf{First case:} We suppose that $Ann(A)\neq \left\{0\right\}$. Then $A$ is an even generalized double extension or an odd generalized double extension of an even-symmetric associative superalgebra by a one-dimensional associative superalgebra with null product. This case was established in \cite{Im} Proposition V.3.\\

\textbf{Second case:} We suppose that $Ann(A)= \left\{0\right\}$. Since $A$ is non-simple, then there exist a minimal graded two-sided ideal $I$ of $A$. The fact that $A$ is B-irreducible such that $dim A> 1$, implies that $I$ is totally isotropic ideal. So, we have $I\subseteq J$. On the other hand, from the non-degeneracy of $B$, there exist a $\Z_{2}$-graded vector space $V$ of $A$ such that $A= J\oplus V$ and ${B\mid}_{I\times V}$ is non-degenerate. According to Lemma \ref{quotient} and since $A/J \cong V$ as associative superalgebras, we obtain that $V$ is a simple associative superalgebras and so $V\in \left\{M_{r,s}(\K), Q_n(\K)\right\}$. Now, applying Theorem \ref{descriptionnov}, we deduce that $A$ is isomorphic to the generalized double extension of $(A_1, Q)$ by $V$, where $V\in \left\{M_{r,s}(\K), Q_n(\K)\right\}$, by means of $(\mu,\lambda,\gamma)$ such that 
\begin{eqnarray*}
\mu &:& V \rightarrow End(J/I)\ \ ;\ \ \mu(v)(x+I) = v.x + I\\
\la &:& V\times V \rightarrow  J/I\ \ ;\ \  \la(v,w) = v . w + I \\
\gamma &:& V\times V \rightarrow {V}^{*} \ \ ; \ \ \gamma (u,v) (w) = B(u . v , w). 
\end{eqnarray*} 
\end{proof}

The following corollary is an immediate consequence of the previous proposition.

\begin{corollary}
Let $(A,B)$ be an even-symmetric associative superalgebra. If $A\notin {\cal O}$, then $A$ is obtained from a finite number of element of ${\cal O}$ by a finite sequence of generalized double extensions by element of $\left\{Q_n(\K), M_{r,s}(\K)\right\}$ and/or even generalized double extensions by the one-dimensional algebra with null product and/or odd generalized double extension by the one-dimensional superalgebra with null even part and/or orthogonal direct sums.\\
\end{corollary}

\subsection{Inductive description of odd-symmetric associative superalgebras}

We give inductive description of odd-symmetric associative superalgebras by using the notion of generalized double extension of odd-symmetric associative superalgebras. To be done, we start by introducing the generalized double extension of odd-symmetric associative superalgebras. Let $(W,*,B)$ be an odd-symmetric associative superalgebra, $(V,\star)$ an associative superalgebra which is not necessary odd-symmetric, $\mu:V \rightarrow {End(W)}$ an even linear map and  $\lambda :V\times V \rightarrow W$ an even bilinear map such that:

\begin{enumerate}[(1)]
\item \label{eq1} $\mu (v)( x*y) = \mu(v)(x)*y$,
\item \label{eq2} $ \mu(v)\circ {\mu (v')}^* = {\mu (v')}^*\circ \mu (v)$,
\item \label{eq3} $ (\mu (v)\circ \mu (v'))(x) = \la(v,v')*x + \mu(v\star v')(x)$,
\item \label{eq4} $\la(v,v'\star v'') + \mu(v)(\la(v',v'')) = {\mu (v'')}^*(\la(v, v')) + \la(v''\star v, v')$, $\forall x, y\in W$ et $ v, v' \in V$,
\end{enumerate}

where ${\mu(v)}^*$ is the adjoint map of $\mu(v)$ with respect to $B$. Next, let us consider the $\Z_{2}$-graded vector space $P(V^*)$ such that ${(P(V^*))}_{\bar{0}}:= {V^*}_{\bar{1}}\ \ \mbox{et}\ \  {(P(V^*))}_{\bar{1}}:= {V^*}_{\bar{0}}$, where $V^*$ is the dual space of $V$. By a simple computation, we can see that $P(V^*)$ have a structure of $V$-bimodule by means of $(l^*,r^*)$ such that:
  
  $$l^*: V\longrightarrow Hom(P(V^*))\ \mbox{defined by }\ l^*(v)(f):= (-1)^{\mid v\mid \mid f\mid}f\circ R_v$$ 
$$r^*: V\longrightarrow Hom(P(V^*))\ \mbox{defined by }\ r^*(v)(f):= f\circ L_v,$$
where $v\in V_{\mid v\mid}$ and $f\in {P(V^*)}_{\mid f\mid}$. Moreover, we consider the two following even bilinear map:  

\begin{eqnarray*}
\nu : V\times W \longrightarrow P(V^*) &\ \mbox{defined by}\ & \nu(v,x)(v'):= B(x,\lambda(v',v));\\
\nu' : W\times V\longrightarrow P(V^*) &\ \mbox{defined by}\ &  \nu'(x,v)(v'):= B(x,\lambda(v,v')).
\end{eqnarray*}
Besides, let us consider the following linear map $\psi: W\times W \longrightarrow P(V^*)$ defined by  $\psi(x,y)(v):= B(\mu(v)(x),y)$. Finally, let $\gamma: V\times V \longrightarrow P(V^*)$ be an even bilinear map which satisfies:

\begin{enumerate}[(1)]
\setcounter{enumi}{4}
\item \label{eq5} $\gamma(v, v') (v'') = \gamma(v', v'')(v)$,
\item \label{eq6} $\gamma(v\star v',v'')(v''') + \gamma(v,v')(v''\star v''') - \gamma(v,v'\star v'')(v''') - \gamma(v',v'')(v'''\star v) \\= B(\la(v''',v),\la(v',v'')) - B(\la(v,v'),\la(v'',v'''))$.
\end{enumerate}

We sum all this in the following definition.\\ 

\begin{defi}
The set composed by $\left\{W,B,V, \mu, \la, \gamma\right\}$ where $(W,B)$ is an odd-symmetric associative superalgebra , $V$ is an associative superalgebra, $\mu$, $\lambda$ and $ \gamma$ are three even maps defined as above and which satisfy (\ref{eq1})-(\ref{eq6}) is called a context of generalized double extension of $W$ by $V$.\\
\end{defi}

\begin{theorem}\label{deg}
Let $\left\{W,B,V,\mu,\lambda,\gamma\right\}$ be a context of generalized double extension of the odd-symmetric associative superalgebra $(W,B)$ by the associative superalgebra $V$. On the $\Z_{2}$-graded vector space  $A:= P(V^*)\oplus W\oplus V$, we define the following product by: 
\begin{eqnarray*}
(v+x+f).(v'+y+g) &=& v\star v' + \la(v,v') + \gamma(v,v') + \mu(v)(y) + \nu(v)(y) + {(-1)}^{\mid v\mid \mid g\mid} g\circ R_v \\&+& \mu^*(v')(x) + \nu'(v')(x) + f\circ L_{v'} + x*y + \ps(x,y)
\end{eqnarray*}
 and the following symmetric bilinear form $\tilde{B}$ by: 
 $$\tilde{B}(v+x+f,v'+y+g):= B(x,y) + f(v') + g(v), $$ 
where $(v+x+f)\in {A}_{\mid x\mid }$ and $(v'+y+g)\in {A}_{\mid y\mid}$. $(A, ., \tilde{B})$ is an odd-symmetric associative superalgebra which is called the generalized double extension of $W$ by $V$ by means of ($\mu, \lambda,\gamma$).\\ 
\end{theorem}

In the following, we prove the Theorem \ref{desind} below and which is the converse of Theorem \ref{deg}.\\ 

\begin{lemma}\label{lemmepreuve1} 
Let $(A,B)$ be an odd-symmetric $B$-irreducible non-simple associative superalgebra, $I$ a minimal graded two-sided ideal of $A$ and $J$ its orthogonal with respect to $B$. Then $J/I$ is an odd-symmetric associative superalgebra.\\
\end{lemma}

\begin{proof}
 The fact that $A$ is non-simple $B$-irreducible implies that the two-sided ideal $I$ is totally isotropic ( i.e $I\subseteq J$). It is clear that $J/I$ with product $(x+I)(y+I)= x.y + I,\ \ \forall x,y\in J$ is an associative superalgebra. In addition, if we consider the following bilinear form $Q$ on $J/I$ defined by $$Q(x+I, y+I):= B(x,y),\ \  \forall x,y\in J,$$ then, we can check easily that it is non-degenerate. So, we deduce that $(J/I,Q)$ is an odd-symmetric associative superalgebra.\\
\end{proof}

\begin{lemma}\label{lemmepreuve2}
Let $(A,B)$ be an odd-symmetric $B$-irreducible non-simple associative superalgebra, $I$ a minimal graded two-sided ideal $A$ and $J$ its orthogonal with respect to $B$. If we suppose that $A=J\oplus V$, where $V$ is a $\Z_{2}$-graded vector space of $A$ and $W:= I\oplus V$. Then $W^{\bot}$, the orthogonal of $W$ with respect to $B$, admits a structure of an odd-symmetric associative superalgebra which is isomorphic to the odd-symmetric associative superalgebra $(J/I,Q)$.\\ 
\end{lemma}

\begin{proof}
By a simple computation, we check that $W$ is a non-degenerate $\Z_{2}$-graded sub-space vector of $A$. Hence, $A:= W\oplus W^{\bot}= I\oplus W^{\bot}\oplus V$ and $B_{W^{\bot}}:= B_{\mid W^{\bot}\times W^{\bot}}$ is non-degenerate. Let $x,y\in W^{\bot}$, as $W^{\bot}\subseteq J$, then $x.y= \alpha(x,y)+ \beta(x,y)$, where $\alpha(x,y)\in W^{\bot}$ and $\beta(x,y)\in I$. Using the associativity of $A$, we show easily that $(W^{\bot}, \alpha)$ is an associative superalgebra and $B_{W^{\bot}}$ define an odd-symmetric structure on $W^{\bot}$. Consequently $(W^{\bot}, \alpha, B_{W^{\bot}})$ is an odd-symmetric associative superalgebra and the restriction of the surjection $s: J\longrightarrow J/I$ to $W^{\bot}$ is an isomorphism of associative superalgebras .\\
\end{proof}

\begin{theorem}\label{desind}
Let $(A,B)$ be an odd-symmetric $B$-irreducible non-simple associative superalgebra, $I$ a minimal graded two-sided ideal of $A$ and $J$ its orthogonal with respect to $B$. We suppose that $A=J\oplus V$, where $V$ is a $\Z_{2}$-graded vector space of $A$, then $A$ is isomorphic to the generalized double extension of the odd-symmetric associative superalgebra $J/I$ by $V$ by means of $(\mu, \lambda, \gamma)$.\\
\end{theorem}

\begin{proof}
Following Lemma \ref{lemmepreuve1} and Lemma \ref{lemmepreuve2}, we have $A= I\oplus W^{\bot}\oplus V$ where $W^{\bot}$ is the orthogonal of $W:= I\oplus V$ such that $W$ and $J/I$ are isomorphic. So to prove the theorem it remains to define a context of generalized double extension of $J/I$ by $V$. For this reason, we consider the following three even maps
 \begin{eqnarray*}
\mu &:& V \rightarrow End(J/I)\ \ ;\ \ \mu(v)(x+I) = v.x + I\\
\la &:& V\times V \rightarrow  J/I\ \ ;\ \  \la(v,w) = v . w + I \\
\gamma &:& V\times V \rightarrow P({V}^{*}) \ \ ; \ \ \gamma (u,v) (w) = B(u . v , w). 
\end{eqnarray*} 
By the associativity of $A$ and the associativity of $B$, we check easily that $\left\{J/I, Q, V \mu, \la, \gamma\right\}$ form a context of generalized double extension of the odd-symmetric associative superalgebra  $J/I$ by $V$ and consequently we can consider $\widetilde{A}=P(V^*)\oplus J/I\oplus V$ the generalized double extension of $(A,B)$ by $V$ by means of $(\mu,\la,\gamma)$. Now, consider the following linear map 
\begin{eqnarray*}
\Delta : A &\longrightarrow& \widetilde{A}\\
i+x+v &\longrightarrow& B(i,.) + s(x)+v.
\end{eqnarray*}
We can check easily that $\Delta$ is an isomorphism of odd-symmetric associative superalgebras such that $\widetilde{B}(\Delta(x),\Delta(y))= B(x,y)$.\\
\end{proof}

\begin{proposition}\label{irre}
Let $(A,B)$ be an odd-symmetric $B$-irreducible non-simple associative superalgebra, then $A$ is either a generalized double extension of an odd-symmetric associative superalgebra by an element of $\left\{M_{r,s}(\K), Q_n(\K),\  r\geq 1, s\geq 0, n\geq 1\right\}$ or a generalized double extension of an odd-symmetric associative superalgebra by a one-dimensional superalgebra with null product.\\
\end{proposition}

 \begin{proof} 
 We consider the following two cases:\\
  
 \textbf{First case:} We suppose that $Ann(A)\neq \left\{0\right\}$ and we consider $I=\K e$ where $e\in Ann(A)\setminus \left\{0\right\}$. We denoted by $J$ the orthogonal of $I$ with respect to $B$. The fact that $B$ is non-degenerate implies that there exist a homogeneous element $d\in A$ such that $A= J\oplus \K d$ and $B(e,d)\neq 0$. According to Lemma \ref{quotient}, we have $A/J$ is a one-dimensional superalgebra with null product. Hence, we deduce that $V:=\K d$ is a one-dimensional superalgebra with null product. Now, applying Theorem \ref{desind}, we obtain that $A$ is a generalized double extension of the odd-symmetric associative superalgebra $(J/I,Q)$ by $V:= \K d$, where $V$ is with null product.\\
 
 \textbf{Second case:} We suppose that $Ann(A)= \left\{0\right\}$. Since $A$ is non-simple, then there exist a minimal graded two-sided ideal $I$ of $A$. The fact that $A$ is B-irreducible, implies that $I$ is totally isotropic ideal. So, we have $I\subseteq J$. On the other hand, as $B$ is non-degenerate, then there exist $V$ a $\Z_{2}$-graded vector space of $A$ such that $A= J\oplus V$ and ${B\mid}_{I\times V}$ is non-degenerate. According to Lemma \ref{quotient} and since $A/J \cong V$ as associative superalgebras, we obtain that $V$ is a simple associative superalgebras and so $V\in \left\{M_{r,s}(\K), Q_n(\K)\right\}$. Now, applying Theorem \ref{desind}, we deduce that $A$ is a generalized double extension of $(J/I, Q)$ by $V$, where $V\in \left\{M_{r,s}(\K), Q_n(\K)\right\}$.\\ 
 \end{proof}

\begin{theorem}
Let $(A,B)$ be an odd-symmetric associative superalgebra. If $A\notin \left\{ \left\{0\right\}, Q_n(\K), n\geq 1\right\}$, then $A$ is obtained from a finite number of element of $\left\{\left\{0\right\}, Q_n(\K), n\geq 1\right\}$ by a finite sequence of generalized double extensions by element of $\left\{M_{r,s}(\K), Q_n(\K)\right\}$ and/or generalized double extensions by a one-dimensional superalgebra with null product and/or orthogonal direct sums.\\
\end{theorem}

\providecommand{\href}[2]{#2}

\end{document}